\documentclass{article}

\usepackage{etex}

\usepackage{latexsym, a4wide}
\usepackage{amsmath, rotating}
\usepackage{amsmath, rotating, color}
\usepackage{amsfonts,amssymb, amsthm}
\usepackage{amstext,amsthm,amsmath}
\usepackage{color}
\usepackage{amssymb,mathtools}
\usepackage{pictex, natbib}
\usepackage{dsfont} 
\newcommand\unnumberedfootnote[1]{ %
        \let\temp=\thefootnote %
        \renewcommand{\thefootnote}{}%
        \footnote{#1}%
        \let\thefootnote=\temp%
        \addtocounter{footnote}{-1}}

\newcommand{\1}{\mathds{1}}
\newcommand{\E}{\mathbb E}
\newcommand{\Pw}{\mathbb P}

\newcommand{\cC}{\mathcal C}
\newcommand{\cD}{\mathcal D}

\newcommand{\cG}{\mathcal G}
\newcommand{\cX}{\mathcal X}
\newcommand{\cY}{\mathcal Y}
\newcommand{\cZ}{\mathcal Z}
\newcommand{\N}{\mathbb N}
\newcommand{\R}{\mathbb R}

\newtheorem{theorem}{Theorem}
\newtheorem{proposition}{Proposition}[section]
\newtheorem{theorem*}[proposition]{Theorem}
\newtheorem{lemma}[proposition]{Lemma}
\newtheorem{corollary}[proposition]{Corollary}
\newtheorem{definition}[proposition]{Definition}
\theoremstyle{definition}
\newtheorem{remark}[proposition]{Remark}

\numberwithin{equation}{section}
\begin{document}


\newpage


\title{\LARGE Markov branching processes with disasters: extinction,
  survival and duality to $p$-jump processes}

\thispagestyle{empty}

\author{{\sc by  Felix Hermann and Peter Pfaffelhuber} \\[2ex]
  \emph{Albert-Ludwigs University Freiburg} } \date{}


\maketitle
\unnumberedfootnote
{\emph{AMS 2000 subject classification.} 
{\tt 60J80} (Primary) {\tt 60J75, 60F10} (Secondary).}

\unnumberedfootnote {\emph{Keywords and phrases.}  Piecewise
  deterministic Markov process, branching process, branching process
  with disasters, duality of Markov processes, extinction probability,
  survival probability}

\begin{abstract}
  \noindent
  A $p$-jump process is a piecewise deterministic Markov process with
  jumps by a factor of $p$. We prove a limit theorem for such
  processes on the unit interval. Via duality with respect to
  probability generating functions, we deduce limiting results for the
  survival probabilities of time-homogeneous branching processes with
  arbitrary offspring distributions, underlying binomial
  disasters. Extending this method, we obtain corresponding results
  for time-inhomogeneous birth-death processes underlying
  time-dependent binomial disasters and continuous state branching
  processes with $p$-jumps.
\end{abstract}

\section{Introduction}
Consider a population evolving according to a branching process
$\cZ'$. In addition to reproduction events, global events called
disasters occur at some random times (independent of $\cZ'$) that kill
off every individual alive with probability $1-p\in(0,1)$,
independently of each other. The resulting process of population sizes
$\cZ$ will be called a branching process subject to binomial disasters
with survival probability $p$. Provided no information regarding fitness
of the individuals in terms of resistance against disasters, this binomial
approach appears intuitive, since the survival events of single individuals
are iid. Applications span from natural disasters as floods and droughts to
effects of radiation treatment or chemotherapy on cancer cells as well as
antibiotics on populations of bacteria. Also, Bernoulli sampling comes into
mind as in the Lenski Experiment (cf. \citealp{CasanovaEtAl2016}).

\noindent
In the general setting of Bellman-Harris processes with non-lattice
lifetime-distribution subject to binomial disasters,
\cite{KaplanEtAl1975} and \cite{AthreyaKaplan1976} have studied the
almost sure asymptotic behaviour as well as asymptotics of the
expectation of the population size and showed that such processes
almost surely either go extinct or explode, giving necessary and
sufficient conditions for extinction. They also computed the limit of
the age-distribution on the set of explosion. In the special case of
homogeneous birth-death processes with binomial disasters,
\cite{BühlerPuri1989} obtained more explicit results regarding
asymptotics and normalised limit distributions as well as the
distribution of the extinction probability conditioned on the disaster
times. \cite{Bartoszynski1989} added the analysis of extinction
probabilities in a multi-type setting.  Furthermore, \cite{Peng1993},
\cite{Thilaka1998} and \cite{KumarII} studied single-type and
multi-type population models underlying binomial disasters with
survival probabilities depending on the time the last disaster
occurred. These models reflect disasters like earthquakes where
pressure builds up over time and increases severity.  A more general
disaster mechanism in a birth-death-scenario has been discussed by
\cite{Brockwell1982}, \cite{Brockwell1985}, \cite{Pakes1986} and
\cite{Pakes1989}, where the absolute population decline after a
catastrophe follows geometric, uniform or even an arbitrary
distribution independent of the population size. Additionally, the
rate of catastrophes is linear in the population size. For continuous
state branching processes with disasters according to some intensity
measure $\nu$, \cite{Bansaye2013} have studied the probability of
extinction.

We will add to this literature of branching processes with disasters
precise results for the asymptotic extinction probability at late
times. As Theorem~\ref{thm:bp-arb.off.dis} shows, if extinction
occurs, the survival probability decays exponentially at a rate which
has a phase transition. We will also be dealing with the
time-inhomogeneous case (see Theorem~\ref{thm:inhom-bdp}), and
extinction probabilities for continuous state branching processes
(CSBP) with binomial disasters (see Theorem \ref{thm:cont-state}).

\noindent
The main technique we are going to use in our study is duality. Recall
that duality of branching systems to a solution of a differential
equation has particularly proven useful for continuous state branching
processes and measure-valued processes. (See \citealp{Etheridge2001}
for an overview and the beginning of Section \ref{sec:bps} for a brief
introduction to this notion of duality.) Bringing this notion back to
a branching process in continuous time (without disasters) $\cZ$,
where every individual branches at rate $\lambda$ and has offspring
distribution with probability generating function (pgf) $h$, the
distribution at time $t$ can be computed via the duality relation
$$ \mathbb E[x^{Z_t}|Z_0=z] = X_t^{z},$$ where $X_0=x$ and $X_t$
solves
\begin{align}
  \label{eq:ODE}
  \dot X = - \lambda (X - h(X))
\end{align}
(cf.\ \citealp{AthreyaNey1972}, Chapter III.3). We will generalize
this equation in order to include binomial disasters. Here, the dual
process will be a piecewise deterministic Markov process (PDMP) $\cX$
on $[0,1]$, where $1-\cX$ evolves according to \eqref{eq:ODE} and
jumps by a factor of $p$ at the rate of the disasters. Such processes
will be called $p$-jump processes below. In Theorem
\ref{thm:convergence-$p$-jump-pros} we will give general limit results
for these processes, which become more precise and concise under
concavity conditions shown in Corollary \ref{cor:p-jump-concave}.
These findings will then translate into limit and asymptotics results
for survival and extinction probabilities of branching processes with
binomial disasters in Theorem \ref{thm:bp-arb.off.dis}.

\noindent
The results of Theorem~\ref{thm:bp-arb.off.dis} will be expanded in
two directions. First, we are dealing with the time-inhomogeneous case
in Theorem~\ref{thm:inhom-bdp}, i.e.\ branching rate, pgf and disaster
rate may depend on time. Here -- using results of \cite{Kendall1948}
for the case without disasters -- we are able to derive the pgf of a
binary branching process subject to binomial disasters, conditioned on
the times of disasters -- similar to the approach of
\cite{BühlerPuri1989} for the homogeneous case. In this case, we also
give in Proposition \ref{prop:inhom-limit} some limits of
pgfs. Second, we apply the duality technique to continuous state
branching processes (CSBP) with binomial disasters. Here, we derive in
Theorem \ref{thm:cont-state} limit results for the extinction
probabilities.

~

\noindent
The manuscript is organised as follows. In Section~\ref{S:res}, we
give our main results on $p$-jump processes with Theorem
\ref{thm:convergence-$p$-jump-pros} and Corollary \ref{cor:p-jump-concave}.
The results on (time-homogeneous) branching processes with disasters are
collected in Theorem~\ref{thm:bp-arb.off.dis}. The case of a birth-death
process, i.e.\ a binary branching process, are given in
Corollary~\ref{cor:results-homog-bd-proc} and extended further to the
time-inhomogeneous case in Theorem~\ref{thm:inhom-bdp}. The CSBP with
disasters is treated in Theorem \ref{thm:cont-state}. In
Section~\ref{sec:pdmps}, we will prove Theorem
\ref{thm:convergence-$p$-jump-pros} and Corollary \ref{cor:p-jump-concave}.
The duality of branching processes with disasters and $p$-jump PDMP will be
established in Section~\ref{sec:bps}, where we will also prove
Theorem~\ref{thm:bp-arb.off.dis}. For the time-inhomogeneous case, we
first need in Section~\ref{sec:reg-var} some results on regularly
varying functions, as collected in Theorem
\ref{thm:asymptotics-of-D-integral}, which might be of interest in their
own right. The proof of Theorem~\ref{thm:inhom-bdp} is then given in
Section~\ref{sec:pr-thm3}.

\section{Results}\label{S:res}

\subsection{$p$-jump Processes}

Let us begin by clarifying the notion of $p$-jump processes.

\begin{definition}\label{def:p-jump-process}\ 
  Let $I=[0,\upsilon]$ for $\upsilon>0$ or $I=\R^+:=[0,\infty)$. Then, let $\alpha:I\to\R$
  with $\alpha(0)\geq0$ and $\alpha(\upsilon)\leq0$, if $I=[0,\upsilon]$. Furthermore, let
  $p\in[0,1]$ and $\cX$ be a right-continuous Markov process on $I$, that performs unit-rate
  jumps from a state $x$ to $px$ and between jumps fulfils $\dot X_t=\alpha(X_t)$. Such a
  process has generator
  \begin{align}
    \cG_\cX f(x)
     &= f(px) - f(x) + \alpha(x)f'(x)\label{eq:generator-of-p-jump-process}
  \end{align}
  for $f\in\cC^1(I)$ and is called a \emph{$p$-jump process with drift
  $\alpha$ on $I$}.
\end{definition}

\begin{remark}\label{rem:def-p-jump-pros}\ 
  \begin{enumerate}
    \item
      Note that such a process is uniquely characterised by $\alpha$, whenever
      $\dot X_t=\alpha(X_t)$ has a unique solution on $I\cap[\varepsilon,\infty)$
      for every $\varepsilon>0$. To ensure this, we will use the Lipschitz-continuity
      conditions $(C_1)$ and $(C_2)$ in Theorem \ref{thm:convergence-$p$-jump-pros}.
      The bounds of $\alpha(0)$ and $\alpha(\upsilon)$ guarantee that the process does
      not leave the interval $I$, such that $\cX$ is well-defined.
    \item
      Due to its multiplicative jumps, such a process can only have $0$ as an
      absorbing state. This happens, iff $\alpha(0)=0$.
    \item
      $1$-jump processes are deterministic, since their jumps have no effect,
      while $0$-jump processes always jump to 0. These two special cases will
      be left aside in Theorem \ref{thm:convergence-$p$-jump-pros} but considered
      in the following Corollary \ref{cor:p-jump-concave} for concave $\alpha$,
      where a more concise conclusion is possible.
  \end{enumerate}
\end{remark}

\noindent
At first, we present our most general limit results for $p$-jump processes.

\begin{theorem}[Convergence of $p$-jump-processes]\label{thm:convergence-$p$-jump-pros}
  Let $I$ and $\alpha$ be as in Definition \ref{def:p-jump-process}, $p\in(0,1)$
  and $X_0\in I$. Also, suppose if $I=\R^+$ that $s_\alpha:=\sup\{x:\alpha(x)>0\}<\infty$.
  Furthermore, let $\alpha'_0:=\lim_{x\to0}\frac1x\alpha(x)$ and assume that
  $\alpha$ satisfies one of the following:
  \begin{itemize}
    \item[$(C_1)$]
      $\alpha$ is Lipschitz-continuous on $I$ or
    \item[$(C_2)$]
      $\alpha$ is Lipschitz-continuous on $I\cap[\varepsilon,\varepsilon^{-1}]$ for
      every $\varepsilon\in(0,1)$ and $\alpha'_0=\infty$.
  \end{itemize}
  Then, there is a $p$-jump process $\cX$ with drift $\alpha$ on $I$ starting in $X_0$,
  such that, letting $\hat\alpha=\sup_{x\in I\setminus\{0\}}\tfrac1x\alpha(x)$, the following
  statements hold:
  \begin{enumerate}
    \item
      If $\hat\alpha<\log\tfrac1p$, then $\mathbb P(X_t\xrightarrow{t\to\infty}0)=1$.
      Additionally, for the $k$th moment of $X_t$, $k\geq1$, the following estimates hold:
      \begin{enumerate}
        \item[$(U_1)$]
          In general (i.e. even for all $\hat\alpha\in\R$),
          \begin{align*}
            \limsup_{t\to\infty} \tfrac 1t \log\E[X_t^k]
              \leq -(1 - p^k - \hat\alpha k).
          \end{align*}
        \item[$(U_2)$]
          If $\hat\alpha > p^k\log\tfrac1p$, letting $\lambda:=\frac1{\hat\alpha}\log\frac1p$,
          we obtain the stronger bound
          \begin{align*}
            \limsup_{t\to\infty} \tfrac 1t \log\E[X_t^k]
              \leq -\big(1 - \tfrac1\lambda(1+\log\lambda)\big).
          \end{align*}
        \item[$(L_1)$]
          If there are $\delta\in\R$ and $\vartheta>0$ such that $\alpha(x)\geq\delta x-\vartheta x^2$
          for all $x\in I$, such that $\delta\leq p^k\log\tfrac1p$, then
          \begin{align*}
            \liminf_{t\to\infty} \tfrac 1t \log\E[X_t^k]
              \geq -(1 - p^k - \delta k).
          \end{align*}
        \item[$(L_2)$]
          If $\delta$ in $(L_1)$ can be chosen positive, letting $\gamma:=\frac1\delta\log\frac1p$,
          we obtain
          \begin{align*}
            \liminf_{t\to\infty} \tfrac 1t \log\E[X_t^k]
              \geq -\big(1 - \tfrac1\gamma(1+\log\gamma)\big),
          \end{align*}          
          which is a stronger bound than $(L_1)$, if $\delta>p^k\log\tfrac1p$.
      \end{enumerate}
    \item
      If $\alpha'_0\in(\log\tfrac1p,\infty]$, let $x_\alpha=\min\{x\in I\setminus\{0\}:\alpha(x)=0\}$.
      Then, $\cX$ converges weakly and its limit $X_\infty$ satisfies $\Pw(X_\infty\in(0,x_\alpha])=1$,
      $\E[X_\infty^{-1}\alpha(X_\infty)]=\log\tfrac1p$
      and for every $k\geq1$
      \[
        \E[X_\infty^k]
         = \frac k{1-p^k}\E[X_\infty^{k-1}\alpha(X_\infty)].
      \]
      Also, the distribution of $X_\infty$ is the unique stationary distribution and
      for every bounded and measurable function $f:I\to\R$ almost surely
      \begin{align*}
        \lim_{t\to\infty}\frac1t\int_0^tf(X_s)ds
         &= \E[f(X_\infty)].
      \end{align*}
  \end{enumerate}
\end{theorem}

\begin{remark}\label{rem:thm-PDMPs}\ 
  \begin{enumerate}
  \item Note that in Theorem
    \ref{thm:convergence-$p$-jump-pros}.\emph1 $\alpha'_0$ is finite,
    such that $\alpha(0)$ has to be $0$ and it holds
    $\alpha'_0=\alpha'(0)$. Hence the similar notation. Case \emph{2.}
    accounts for both possibilities $\alpha(0)=0$ with
    $\alpha'(0)=\infty$ as well as $\alpha(0)>0$.
  \item Although the theorem shows existence of a $p$-jump process on
    $\R^+$, it is notable that such a process will only assume values
    in $[0,\max\{X_0,s_\alpha\}]$, since by definition it can never
    grow beyond $\sup(\{X_0\}\cup\{x:\alpha(x)>0\})$.
  \end{enumerate}
\end{remark}

\noindent
In the case where $\alpha$ is concave, the bounds $(U_i)$ and $(L_i)$ align and
the continuity conditions become evident such that we can give a much more concise
result. Here, we will also include the cases $p=0$ and $p=1$.

\begin{corollary}\label{cor:p-jump-concave}
  Let $p\in[0,1]$, $I$ and $\alpha$ be as in Definition \ref{def:p-jump-process}.
  Additionally assume that $\alpha$ is concave, $\alpha''(0)\in[-\infty,0]$ exists
  and that either
  \begin{itemize}
  \item[--] $I=[0,\upsilon]$ and $\alpha'(\upsilon)>-\infty$ or
    \item[--] $I=\R^+$ and there is an $x>0$ such that $\alpha(x)=0$.
  \end{itemize}
  Then, letting $X_0\in I$ and $\alpha'_0:=\lim_{x\to0}\frac1x\alpha(x)$, there is a
  $p$-jump process $\cX$ with drift $\alpha$ on $I$ starting in $X_0$ and satisfying:
  \begin{enumerate}
    \item
      If $\alpha'_0<\log\tfrac1p$ or $p=0$, then $X_t\xrightarrow{t\to\infty}0$ almost surely.
      Also, for $k\geq1$
      \begin{align*}
        \lim_{t\to\infty}-\tfrac1t\log\E[X_t^k]
         &= \begin{cases}
           1 + \max\{0, - k\alpha'_0\} & \text{if }p=0,\\
           1 - p^k - k\alpha'_0 & \text{if }\alpha'_0\leq p^k\log\tfrac1p,\\
           1 - \tfrac1\gamma(1+\log\gamma) & \text{otherwise},
         \end{cases}
      \end{align*}
      where $\gamma=\log(\tfrac1p)/\alpha'_0$.
    \item
      If $\alpha'_0=\log\tfrac1p$, then $\frac1t\log\E[X_t^k]\xrightarrow{t\to\infty}0$
      and $\limsup_tX_t\leq m_\alpha:=\sup\{x\in[0, x_\alpha]:\alpha(x) = x\alpha'_0\}$
      almost surely. In particular, if $\alpha$ is strictly concave on an interval $(0,\varepsilon)$,
      then $X_t\to m_\alpha=0$ almost surely.
    \item
      If $\alpha'_0>\log\tfrac1p$, then $x_\alpha:=\sup\{x:\alpha(x)>0\}\in(0,\infty)$,
      $\cX$ converges weakly and its limit $X_\infty$ satisfies $\mathbb P(X_\infty\in(0,x_\alpha])=1$,
      $\E[X_\infty^{-1}\alpha(X_\infty)]=\log\tfrac1p$ and for every $k\geq1$
      \begin{align*}
        \E[X_\infty^k]
         &= \frac k{1-p^k}\E[X_\infty^{k-1}\alpha(X_\infty)].
      \end{align*}
      Also, the distribution of $X_\infty$ is the unique stationary
      distribution and for every bounded and measurable function
      $f:\R_+\to\R$ almost surely
      \begin{align*}
        \lim_{t\to\infty}\frac1t\int_0^tf(X_s)ds
         &= \E[f(X_\infty)].
      \end{align*}
  \end{enumerate}
\end{corollary}
%

\subsection{Branching Processes with Binomial Disasters}

\noindent
Applying Corollary \ref{cor:p-jump-concave} to a $p$-jump process, via duality with respect
to probability generating functions in Theorem \ref{thm:bp-arb.off.dis} we obtain
immediate limit results for the following class of branching processes with binomial
disasters.

\begin{definition}\label{def:hom-bp-abr-off-dis}
  Let $\lambda>0$, $q=(q_k)_{k\geq0}$ a distribution on $\N_0$, $\kappa>0$ and $p\in[0,1]$.
  A Markov process $\cZ$ on $\N_0$ with generator
  \begin{align*}
    \cG_\cZ f(z)
     &= \lambda z\sum_{k\geq0}q_k\big(f(z-1+k)-f(z)\big)
       + \kappa\sum_{k=0}^z\genfrac(){0pt}{}zkp^k(1-p)^{z-k}\big(f(k)-f(z)\big)
  \end{align*}
  for $f \in \mathcal B(\N_0)$, the set of real-valued, bounded
  functions on $\N$, is called a \emph{homogeneous branching process
    with death-rate $\lambda$ and offspring distribution $(q_k)$,
    subject to binomial disasters at rate $\kappa$ with survival
    probability $p$} and will be denoted by
  $\cZ^h_{\lambda,q,\kappa,p}$.
\end{definition}

\noindent
Such a process describes the size of a population that behaves in the following way:
Every individual dies with rate $\lambda$ and leaves behind a random number of
offsprings distributed according to $(q_k)$. Independent of this growth mechanism,
with rate $\kappa$ global events occur that kill off every individual alive at that
time with probability $1-p$ independently of each other.

\begin{theorem}\label{thm:bp-arb.off.dis}
  Let $\lambda>0$, $q=(q_k)_{k\geq0}$ a distribution on $\N_0$ with
  expectation $\mu:=\sum_kkq_k$, $\kappa>0$ and $p\in[0,1)$. Then,
  if $p=0$, $\cZ$ goes extinct almost surely with
  \[
    \lim_{t\to\infty}-\tfrac1t\log\mathbb P(Z_t>0)
      = \kappa + \max\{\lambda(1-\mu),0\}.
  \]
  Otherwise, letting $\nu=\lambda(\mu-1)/(\kappa\log\frac1p)$,
  $\cZ:=\cZ^h_{\lambda,q,\kappa,p}$ satisfies
  \begin{enumerate}
    \item
      if $\nu\leq p$, $\cZ$ goes extinct almost surely and
      \begin{align*}
        \lim_{t\to\infty}-\tfrac1t\log\mathbb P(Z_t>0)
         &= (1-p)\kappa - \lambda(\mu-1).
      \end{align*}
    \item
      if $p<\nu\leq1$, $\mathcal Z$ goes extinct almost surely and
      \begin{align*}
        \lim_{t\to\infty}-\tfrac1t\log\mathbb P(Z_t>0)
         &= 
            \kappa(1 - \nu - \nu\log(\tfrac1\nu)).
      \end{align*}
    \item
      if $\nu>1$, then $\cZ$ survives with positive probability, where
      \begin{align}
        0
          < \mathbb P(\lim_{t\to\infty}Z_t=\infty)
          = 1 - \mathbb P(\lim_{t\to\infty}Z_t=0)
          = \sum_{k=1}^z\genfrac(){0pt}{}zk(-1)^{k-1}\E[X^k]
          < 1-x_\ast^{z_0},
          \label{eq:bp-arb.off.dis-surv.prob}
      \end{align}
      where $h(x)=\sum_kx^kq_k$ is the pgf of $q$, $x_\ast$ is the
      smallest fixed point of $h$ and $X$ is a random variable on $(0,1-x_\ast]$
      satisfying $\E[X^{-1}(1-h(1-X))]=1+\tfrac\kappa\lambda\log\tfrac1p$
      and
      \begin{align}
        \E[X^k]
           = \frac{\lambda k}{\lambda k + \kappa(1-p^k)}\E[X^{k-1}(1-h(1-X))].
           \label{eq:bp-arb.off.dis-surv.prob.rec}
      \end{align}
  \end{enumerate}
\end{theorem}

\begin{remark}\ 
  \begin{enumerate}
    \item
      Rearranging the inequalities in terms of $\mu$, we obtain
      \[
        \textit{1. if }\mu\leq 1 + \tfrac{\kappa p}\lambda\log\tfrac1p,\quad
        \textit{2. if }1 + \tfrac{\kappa p}\lambda\log\tfrac1p<\mu\leq1 + \tfrac\kappa\lambda\log\tfrac1p,\quad
        \textit{3. if }\mu>1 + \tfrac\kappa\lambda\log\tfrac1p.
      \]
      These give insight into how supercritical the underlying
      branching process has to be in order to
      survive the disasters.\\
      Also, this formulation illustrates the continuity of the theorem
      in $p=1$: Classical results for such processes without disasters
      (cf. \citealp[Theorem 11.1, p.109]{Harris1963}) show that $\cZ$
      goes extinct almost surely if $\mu\leq1$ with
      $-\frac1t\log\Pw(Z_t>0)\to\lambda(1-\mu)$ as $t\to\infty$, which
      aligns with \emph1., while for $\mu>1$,
      $\Pw(Z_t\to\infty)=1-\Pw(Z_t\to0)=1-x_\ast^{z_0}$, which is the
      upper bound for the survival probability in
      \eqref{eq:bp-arb.off.dis-surv.prob} for $p<1$.
    \item
      While \cite{KaplanEtAl1975} have already shown the almost sure extinction in \emph{1.}
      and \emph{2.} as well as the fact in \emph{3.} that $\cZ$ almost surely either goes extinct
      or explodes, we offer an alternative proof via our duality results plus rates of convergence
      for the survival probability including the case $\mu=\infty$.
      Also, making use of \eqref{eq:bp-arb.off.dis-surv.prob.rec}, our result offers a way to
      compute the exact extinction probability in \emph{3}. Since the recursion in
      \eqref{eq:bp-arb.off.dis-surv.prob.rec} depends on the offspring pgf $h$, in general this
      formula can be difficult to compute. Corollary \ref{cor:results-homog-bd-proc}, however,
      shows that in the example of homogeneous birth-death-processes it is feasible.
  \end{enumerate}
\end{remark}

\noindent
The following corollary applies Theorem \ref{thm:bp-arb.off.dis} to birth-death-processes
with disasters. This does not only provide a nice transition to the next theorem, but offers
an example where (using the relation \eqref{eq:bp-arb.off.dis-surv.prob.rec}) we can
explicitly compute the survival probability.

\begin{corollary}\label{cor:results-homog-bd-proc}
  Let $\cZ:=(Z_t)_t$ be a homogeneous birth-death-process with respective rates $b>0$ and $d\geq0$
  that underlies binomial disasters at a rate of $\kappa>0$ with survival probability $p\in(0,1)$.
  \begin{enumerate}
    \item
      If $b-d\leq \kappa p\log\tfrac1p$, $\cZ$ goes extinct almost surely and
      \begin{align*}
        \lim_{t\to\infty}-\tfrac1t\log\Pw(Z_t>0)
         &= (1 - p)\kappa - (b-d).
      \end{align*}
    \item
      If $\kappa p\log\tfrac1p<b-d\leq \kappa\log\tfrac1p$, $\cZ$ goes extinct almost surely and
      \begin{align*}
        \lim_{t\to\infty}-\tfrac1t\log\Pw(Z_t>0)
         &= \kappa - \frac{b-d}{\log\frac1p}\Big(1 + \log\Big(\frac{\kappa\log\frac1p}{b-d}\Big)\Big).
      \end{align*}
    \item
      If $b-d>\kappa\log\tfrac1p$, then
      $
        \Pw_k(\lim_{t\to\infty}Z_t=0)
          + \Pw_k(\lim_{t\to\infty}Z_t=\infty)
          = 1
      $ and
      \begin{align*}
        \Pw_k(\lim_{t\to\infty}Z_t=\infty)
         &= \Big(1-\frac{d+\kappa\log\frac1p}b\Big)\sum_{\ell=1}^k\genfrac(){0pt}{}k\ell(-1)^{\ell-1}
              \prod_{m=1}^{\ell-1}\Big(1-\frac{dm+(1-p^m)\kappa}{bm}\Big).
      \end{align*}
  \end{enumerate}
\end{corollary}
\begin{proof}
  First note that $\cZ$ is a $\cZ^h_{\lambda,q,\kappa,p}$-process with
  $\lambda=b+d$, $q_0=d/(b+d)$ and $q_2=b/(b+d)=1-q_0$. Thus,
  $\lambda(\mu-1)=(b+d)(q_2-q_0)=b-d$ and
  $\nu=(b-d)/(\kappa\log\frac1p)$, already providing \emph{1.} and
  \emph2. by insertion in Theorem \ref{thm:bp-arb.off.dis}.  For
  \emph{3.} we derive $1-h(1-x)=1-q_0-q_2(1-x)^2=q_2x(2-x)$, yielding
  a simple recursion in \eqref{eq:bp-arb.off.dis-surv.prob.rec}
  concluding the proof.
\end{proof}

\noindent
In Section \ref{sec:pr-thm3} we will develop tools for the analysis of inhomogeneous
birth-death processes with time-dependent disasters, generalising the setting of Corollary
\ref{cor:results-homog-bd-proc}. For this, mind the following definition, where the
birth, death and disaster rates $b,d,\kappa$ as well as the survival probability $p$
are now given as functions of $t$.

\begin{definition}\label{def:inhom-bd-w-dis}
  Let $b,d,\kappa: \R_+\to \R_+$ and $p:\R_+\to[0,1]$, where we
  abbreviate $b_s := b(s), d_s := d(s), \kappa_s :=\kappa(s)$ and
  $p_s := p(s)$. A Markov process $\cZ$ on $\N_0$ with time-dependent
  generator (see Section 4.7A of \citealp{EK86})
  \begin{align*}
    \cG_{\cZ,t} f(z)
    &= b_tz\big(f(z+1)-f(z)\big) + d_tz\big(f(z-1)-f(z)\big)\\[1em]
    &\qquad +\kappa_t\sum_{k=0}^z\genfrac(){0pt}{}zkp_t^k(1-p_t)^{z-k}\big(f(k)-f(z)\big)
  \end{align*}
  for $t\geq 0$ and $f\in \mathcal B(\N_0)$, is called an
  \emph{inhomogeneous birth-death-process with birth-rate $b$ and
    death-rate $d$, subject to binomial disasters with survival
    probability $p$ occurring at rate $\kappa$} and will be denoted by
  $\cZ^{in}_{b,d,\kappa,p}$.
\end{definition}

\noindent
Key to our approach is Lemma \ref{lem:inhom-duality}, which
computes the conditioned pgf delivering some kind of \emph{stronger} duality, enabling us
with Proposition \ref{prop:inhom-limit} to easily give pgf limit results in terms of that
dual process. While these tools offer room for further generalisation (cf. Remark
\ref{rem:prop:inhom-limit}), we give the following theorem as an example of application,
where we also make use of Theorem \ref{thm:asymptotics-of-D-integral}.

\begin{theorem}\label{thm:inhom-bdp}
  Let $b,d,\kappa$ be non-negative right-continuous functions on
  $\R_+$ with left limits and $p: \R_+\to[0,1]$ left-continuous with
  right limits such that $p_t=0$ only if $\kappa_t=0$ and, letting
  $\Lambda_\kappa^{-1}(t):=\inf\{s>0:\int_0^s\kappa_sds>t\}$, such
  that the map $-\log(p(\Lambda_\kappa^{-1}(\cdot)))$ is regularly
  varying. Furthermore, let $h:\R_+ \to \R_+$ continuous and
  non-decreasing with $\lim_{t\to\infty}t^{-\alpha}h(t)=\infty$ for
  some $\alpha>0$, as well as $\iota\in\{-1,1\}$ such that
  \begin{align*}
    \frac1{h(t)}\int_0^t\Big(b_s - d_s - \kappa_s\log\Big(\frac1{p_s}\Big)\Big)ds
     \xrightarrow{t\to\infty}\iota.
  \end{align*}
  Then, $\cZ:=\cZ^{in}_{b,d,\kappa,p}$ satisfies
  \begin{enumerate}
  \item if $\iota=1$ and for some $\varepsilon>0$ holds
    $\displaystyle\int_0^\infty e^{-(1-\varepsilon)h(s)}b_sds<\infty$,
    then $\Pw(Z_t\xrightarrow{t\to\infty}0)<1$.
  \item if $\iota=-1$ or for some $\varepsilon>0$ holds
    $\displaystyle\int_0^\infty e^{-(1+\varepsilon)h(s)}b_sds=\infty$,
    then $\Pw(Z_t\xrightarrow{t\to\infty}0)=1$.
  \end{enumerate}
\end{theorem}

\begin{remark}\ 
  \begin{enumerate}
    \item
      For $h(t)=t$ and $b,d,\kappa,p$ constant, this result aligns with the homogeneous
      case (cf. Corollary \ref{cor:results-homog-bd-proc}).
    \item
      The regular variation condition on $p$ and $\kappa$ is equivalent to the
      existence of $\beta\in\R$ and some slowly varying function $\ell$, such
      that
      $
        p(t)
          = \exp(-\Lambda_\kappa(t)^\beta\ell(\Lambda_\kappa(t)).
      $
      We need $t^{-\alpha}h(t)\to\infty$ for some $\alpha>0$ to handle the case
      $\beta=-1$ in which Theorem \ref{thm:asymptotics-of-D-integral} is
      inconclusive. If $\beta\neq-1$, one may choose $\alpha=0$.
    \item
      The condition that $\kappa=0$ whenever $p=0$ ensures that no \emph{terminal}
      disasters occur, i.e. disasters that render $\cZ$ extinct with probability 1.
      Dropping this, $\int_0^t\kappa_s\log(1/p_s)ds$ might no longer be finite.
      However, letting $\kappa^-_t:=\kappa_t\cdot\delta_{0,p_t}$ and
      $\kappa^+_t:=\kappa_t-\kappa^-_t$, it is possible to apply Theorem
      \ref{thm:inhom-bdp} to the process $\cZ=\cZ^{in}_{b,d,\kappa^+,p}$ without the
      terminal disasters and separately compute the probability $\pi_{term}$ that at
      least one terminal disaster occurs, which for a unit-rate Poisson process
      $(P_t)$ satisfies
      \begin{align*}
        \pi_{term}
         &= \Pw(P_{\int_0^\infty\kappa^-_tdt}>0)
          = 1 - \exp\Big(-\int_{t:\,p(t)=0}\kappa_t\,dt\Big).
      \end{align*}
      Since the sets $\{\kappa^+>0\}$ and $\{\kappa^->0\}$ are disjoint, the
      respective counts of disasters on these sets are independent. This implies
      that the terminal disasters only affect the positivity of the survival
      probability, if $\int_{t:\,p(t)=0}\kappa_t\,dt=\infty$.
    \item The cases where a normalisation function $h$ as in Theorem
      \ref{thm:inhom-bdp} does not exist, are discussed in
      Remark~\ref{rem:inhom-conv-rates}.1. Rates of convergence for
      the survival probability in case \emph2. will briefly be
      discussed in Remark \ref{rem:inhom-conv-rates}.2.
  \end{enumerate}
\end{remark}

\subsection{Continuous State Branching Processes with Binomial Disasters}
The application of $p$-jump processes is not limited to branching
processes with discrete states. We will now discuss survival and
extinction for continuous state branching processes (see e.g.\
\citealp{Lambert2008} for an overview) with binomial disasters.  A
similar model is studied in \cite{Bansaye2013}, where multiplicative
jumps occur for any factor according to some intensity measure. Their
Theorem~1 shows existence and uniqueness of the process we now define.

\begin{definition}\label{def:csbp-w-dis}
  Let $N$ a measure on $\mathbb R_+$ with $\int_0^\infty\min(y,y^2)N(dy)<\infty$,
  $b\in\R, c, \kappa \in \R_+$ and $p\in(0,1)$. Then, the $\R_+$-valued Markov
  process with generator
  $$
    \mathcal G_{\mathcal Z}f(z)
      = bzf'(z) + c z f''(z) + \kappa (f(pz) - f(z)) + \int_0^\infty (f(z+y) - f(z) - yf'(z))z N(dy),
  $$
  for $f \in \mathcal C_b^2(\mathbb R_+)$ (the space of bounded, twice
  continuously differentiable functions) is called the \emph{continuous
  state branching process with $p$-disasters} and will be denoted by
  $\cZ^{cs}_{b,c,N,\kappa,p}$.
\end{definition}

\noindent
Viewing a continuous state branching process as a scaling limit of a
discrete branching process, by the law of large numbers the equivalent
of a binomial disaster would be a $p$-jump of the population size,
represented by the term $\kappa(f(pz)-f(z))$ in $\cG_\cZ$.  The next
Theorem is similar to (but less precise than) Corollary~6 of
\cite{Bansaye2013}, but does not need their restrictions
$\int_0^\infty y^2N(dy)<\infty$ and $c>0$.

\begin{theorem}\label{thm:cont-state}
  Let $\cZ = \cZ^{cs}_{b,c,N,\kappa,p}$ be a continuous state branching
  process with $p$-disasters as in Definition \ref{def:csbp-w-dis}
  and let
  $$\alpha: \R_+ \to\R,x\mapsto bx-cx^2 - \int_0^\infty(e^{-xy}-1+xy)N(dy)$$
  satisfy $\limsup_{x\to\infty}x^{-(1+\varepsilon)}\alpha(x)<0$ for
  some $\varepsilon>0$.
  \begin{enumerate}
    \item
      If $b\leq\kappa\log\frac1p$, $Z_t\xrightarrow{t\to\infty}0$ almost surely and
      \begin{align*}
        \lim_{t\to\infty}-\tfrac1t\log\Pw(Z_t>0)
         &= \begin{cases}
              (1-p)\kappa - b & \text{ if }b\leq \kappa p\log\frac1p,\\
              \kappa - \frac\kappa\gamma(1+\log\gamma) & \text{ otherwise,}
            \end{cases}
      \end{align*}
      where $\gamma:= \kappa\log(\frac1p)/b$.
    \item
      If $b>\kappa\log\frac1p$, $\lim_{t\to\infty}\Pw(Z_t=0|Z_0=z)\in(e^{-z\xi},1)$,
      where $\xi$ denotes the largest root of $\alpha$.
  \end{enumerate}
\end{theorem}
\noindent
We will only give an outline of the proof, since its structure is
illustrative for the proofs to come and it differs only in details.

\begin{proof}[Sketch of proof.]
  Due to a rescaling argument, we just need to show the case of $\kappa=1$.
  Setting $H(x,z) := e^{-xz}$ and applying the generator to the function
  $z\mapsto H(x,z)$ for $x$ fixed gives
  \begin{align*}
    \mathcal G_{\mathcal Z} H(x,.)(z)
     &= \Big(
          bx - cx^2 - \int_0^\infty (e^{-xy} - 1 + xy)N(dy)\Big)\frac{\partial H}{\partial x}(x,z)
           + H(xp,z) - H(x,z)
  \end{align*}
  In other words (cf. \eqref{eq:dual-gens} at the beginning of Section \ref{sec:bps}),
  $\mathcal Z$ is dual to the $p$-jump process with drift $\alpha$. Now, note that
  $\xi=\bar x_\alpha<\infty$, since $\limsup_{x\to\infty}x^{-(1+\varepsilon)}\alpha(x)<0$
  and thus, $\alpha(x)\to-\infty$ as $x\to\infty$. Hence, either $c>0$ or $N((0,\infty))>0$.
  In any case, for all $x>0$
  \begin{align*}
    \alpha''(x)
     &= -2c - \int_0^\infty x^2e^{-xy}N(dy)
      < 0
  \end{align*}
  and $\alpha$ is strictly concave. Hence, Corollary \ref{cor:p-jump-concave} applies
  with $\alpha'_0=\alpha'(0)=b$.\\
  Since $\limsup_{x\to\infty}x^{-(1+\varepsilon)}\alpha(x)<0$, $\cX$ comes down from
  infinity in the sense that $\cX$ is well-defined in the limit of $X_0\to\infty$. It
  follows
  \begin{align*}
    \Pw(Z_t=0|Z_0=z)
     &= \lim_{x\to\infty}\E[e^{-xZ_t}|Z_0=z]
      = \E[e^{-zX_t}|X_0=\infty].
  \end{align*}
  Using that $x-\frac{x^2}2\leq 1-e^{-x}\leq x$ for all $x\geq0$ to
  estimate the rates of convergence, Corollary \ref{cor:p-jump-concave}
  concludes the proof, where the almost sure convergence in \emph{1.}
  comes from the fact that $\cZ$ is a supermartingale. (This can be
  seen by computing the martingale given by
  $Z_t\exp(-\int_0^t\frac{\cG_\cZ\text{id}(Z_s)}{Z_s}ds)$.)
\end{proof}

\begin{remark} 
  In the case $b>\kappa\log\frac1p$, we have
  \begin{align*}
    \Pw(Z_t=0|Z_0=z)
    &\xrightarrow{t\to\infty} \E[e^{-zX_\infty}]
      =: L(z),
  \end{align*}
  the Laplace transform of the limit of $\cX$. Using stationarity it follows from
  $\E[\cG_\cX f(X_\infty)]=0$ for $f(x)=e^{-zx}$, that $L$ has to satisfy
  the functional equation
  \begin{align*}
    \kappa(L(pz) - L(z)) - bzL'(z) + czL''(z)
    + z\int_0^\infty L(z+y) - L(z) + y L'(z) N(dy)
    = 0,
  \end{align*}
  which for appropriate $N$ might deliver a more precise result than
  Theorem \ref{thm:cont-state}.\emph2.
\end{remark}

\section{Piecewise Deterministic Markov Processes}\label{sec:pdmps}

In this section, we start by proving Theorem \ref{thm:convergence-$p$-jump-pros},
mainly by applying large deviation results for Poisson processes, given in Appendix
\ref{sec:ldp}, and the work of \cite{BladtNielsen2017} regarding
regenerative processes. Using this, we prove Corollary \ref{cor:p-jump-concave}
in Section~\ref{sec:pr-cor-concave} for concave $\alpha$ offering a more concise result including continuities for
$p\in\{0,1\}$.

\subsection{Proof of Theorem \ref{thm:convergence-$p$-jump-pros}}

\begin{proof}
  
  At the very beginning, suppose that the theorem already holds for $p$-jump processes
  on $[0,1]$, let $\alpha$ be as in the assumptions, $s:=\max\{X_0,s_\alpha\}$
  and consider $a:[0,1]\to\R^+,x\mapsto\alpha(sx)/s$. It is straightforward
  to show that $a$ satisfies the conditions of Definition \ref{def:p-jump-process}
  as well as the ones of Theorem \ref{thm:convergence-$p$-jump-pros} such that there
  is a $p$-jump process $\overline\cX=(\overline X_t)$ with drift $a$ on
  $[0,1]$ starting in $sX_0$ for which the assertions of the theorem hold,
  where $a'_0=\alpha'_0$, $\hat{a}=\hat\alpha$ and $x_a=sx_\alpha$. Also,
  $a(x)\geq\delta x-\vartheta x^2$ for all $x\in[0,1]$ iff
  $\alpha(x)\geq\delta x - \frac\vartheta sx^2$ for all $x\in[0,s]$.
  Considering that $t\mapsto X_t:=s_\alpha\overline X_t$ also performs
  $p$-multiplicative jumps at rate 1 and in betwteen satisfies
  $\dot X_t=s\cdot a(\overline X_t)=\alpha(X_t)$, we obtain that the theorem
  holds. Hence, without loss of generality, we assume for the rest of the proof
  that $I=[0,1]$.\\[.5em]
  Note that in $(C_1)$, $\alpha$ is Lipschitz-continuous on the whole interval $[0,1]$,
  while in $(C_2)$, there is an $x^+>0$ such that $\alpha(x)>0$ for all $x\in(0,x^+]$. In
  the latter case, the initial value problem with $f'=\alpha(f)$ and $f(0)>0$ is
  equivalent to the restriction $f'=\alpha\big|_{[\min\{f(0),x^+\},1]}(f)$, since
  $\alpha(\min\{f(0),x^+\})>0$, where $\alpha$ is Lipschitz-continuous on $[\min\{f(0),x^+\},1]$.
  In either case, by the Picard-Lindel\"of Theorem, for every deterministic piece on an interval
  between two jumps $[\tau_k,\tau_{k+1})$ the initial value problem with $\dot X_t = \alpha(X_t)$
  and $X_{\tau_k}:=pX_{\tau_k-}>0$ has a unique solution. Additionally, the bound $\alpha(1)\leq0$
  ensures that $\cX$ does not leave the interval $[0,1]$. Thus, $\cX$ is well-defined.
  (Also note that, since by definition $\alpha'_0\leq\hat\alpha$, $(C_2)$ only concerns
  case \emph2. Moreover, $\alpha'_0<\infty$ only if $\alpha(0)=0$, in which case
  $\alpha'_0=\alpha'(0)$.)\\[1em]
  \emph{1.} Let $\beta(y):=e^y\alpha(e^{-y})$ and $Y_t:=-\log X_t\in \R_+$. Then, the process
  $\mathcal Y:=(Y_t)_{t\geq0}$ has the generator
  \begin{align*}
    \mathcal G_{\mathcal Y}g(y)
     &= \big(\mathcal G_{\mathcal X}(g\circ(-\log))\big)(e^{-y})\\[1em]
     &= g\big(-\log(pe^{-y})\big) - g(y)
       + \alpha(e^{-y})\cdot\big(-\tfrac1xg'(-\log x)\big)\big|_{x=e^{-y}}\\[1em]
     &= g(y+\log\tfrac1p) - g(y) - \beta(y)g'(y).
  \end{align*}
  The assertion implies that $\beta(y)=\frac1{e^{-y}}\alpha(e^{-y})\leq\hat\alpha<\log\tfrac1p$
  for all $y\in \R_+$. Letting $(P_t)$ be the unit-rate Poisson process jumping
  simultaneously with $\mathcal Y$, it follows that
  \begin{align}
    P_t\cdot\log\tfrac1p - t\hat\alpha
     \leq Y_t\label{eq:bound-Y}
  \end{align}
  for all $t$, since the jumps are identical, the left side always starts in 0 and has
  point-wise inferior drift. The law of large numbers, giving us $\lim_{t\to\infty}P_t/t=1$
  and thus almost surely $\liminf_t(Y_t/t)\geq\log\tfrac1p-\hat\alpha>0$, shows that
  $-\log X_t=Y_t\longrightarrow_{t\to\infty}\infty$ with probability 1.\\[0.5em]
  $(U_1)$ Using the generator of $\cX$ on $f(x)=x^k$, we obtain
  \begin{align}
    \tfrac d{dt}\E[X_t^k]
     &= \E\Big[\big(\cG_{\cX}(\cdot)^k\big)(X_t)\Big]
      = \E[p^kX_t^k - X_t^k + \alpha(X_t)kX_t^{k-1}]
      \leq \E[X_t^k](p^k-1+k\hat\alpha).
      \label{eq:U1}
  \end{align}
  Considering Gronwall's inequality, it is straightforward to deduce $(U_1)$
  (even for arbitrary $\hat\alpha$).\\[0.5em]
  $(U_2)$ For the (stronger) upper bound in the case of $\hat\alpha>p^k\log\tfrac1p$ we reuse
  $\eqref{eq:bound-Y}$ to compute
  \begin{align}
    \E[X_t^k]
      &= \E[e^{-kY_t}]\notag\\[1em]
      &\leq u(t)
       := \E[e^{-k(P_t\log(1/p)-t\hat\alpha)}\wedge 1]\label{eq:u(t)}\\[1em]\notag
     &= \mathbb P\Big(P_t \leq \frac{\hat\alpha}{\log(1/p)}t\Big)
         + \sum_{\ell\geq\hat\alpha t/(\log(1/p))}
             e^{k(\hat\alpha t-\log(1/p)\ell)} e^{-t}\frac{t^\ell}{\ell!}\\ \notag
     &= \underbrace{\mathbb P\Big(P_t \leq \frac t\lambda\Big)}_{=:u_1(t)}
       +\underbrace{e^{-t(1-p^k-\hat\alpha k)}\mathbb P\Big(P_{p^kt}\geq\frac1{p^k\lambda}p^kt\Big)}
                   _{=:u_2(t)}.
  \end{align}
  First, using \eqref{lld:ld2} of Lemma \ref{lem:ld-PP} with $x = \tfrac1\lambda < 1$, we compute
  \begin{align*}
    -\frac1t\log u_1(t)
    &\xrightarrow{t\to\infty}
       1 - \frac1\lambda + \frac1\lambda\log\frac1\lambda
     = 1 - \tfrac1\lambda(1+\log\lambda) =:A_{\lambda}.
  \end{align*}
  For the second term, $u_2$, we see that $p^k\lambda<1$ and thus, by \eqref{lld:ld1} of Lemma
  \ref{lem:ld-PP},
  \begin{align*}
    - \frac 1t \log u_2(t)
      \xrightarrow{t\to\infty}
       \ &1 - p^k - \hat\alpha k
           + p^k\Big(1 - \frac1{p^k\lambda} + \frac1{p^k\lambda}\log\frac1{p^k\lambda}\Big)\\[1em]
      =\ &1 - \hat\alpha k - \frac1\lambda\Big(1 - \log\frac1\lambda - k\log\frac1p\Big)
      = A_{\lambda}.
  \end{align*}
  Combining these two results, we obtain $\lim_{t\to\infty}\tfrac1t\log u(t)=-A_\lambda$,
  which equates to the desired bound.\\[0.5em]
  To recognise that the bound in $(U_2)$ is in fact stronger than the one of $(U_1)$, we need to
  verify that $1-p^k-\hat\alpha k\leq A_\lambda$, if $p^k\lambda<1$. In this case, the function
  $h:x\mapsto x+\tfrac1\lambda\log\tfrac1x$ is strictly decreasing on $[p^{k},\lambda^{-1}]$,
  since $h'(x) = 1 - \tfrac1{\lambda x}$. Thus, inserting $\hat\alpha=\tfrac1\lambda\log\tfrac1p$,
  the difference satisfies
  \begin{align*}
    A_\lambda - (1-p^k - \hat\alpha k)
     &= p^k + \frac1\lambda\log\frac1{p^k} - \frac1\lambda(1+\log\lambda)\\[1em]
     &= h(p^k)-h(\lambda^{-1})
      > 0.
  \end{align*}
  $(L_1)$ and $(L_2)$: Let $w:[0,1]\to\R,x\mapsto\delta x-\vartheta x^2$. Then, $\cX$ is bounded
  below by a process $\mathcal W=(W_t)_t$ with generator
  \begin{align*}
    \mathcal G_{\mathcal W}f(x)
     &= f(px)-f(x) + w(x)f'(x),
  \end{align*}
  if $\cX$ and $\mathcal W$ have equal initial values and are coupled in such a way that they
  jump simultaneously at the jump times of a Poisson process $(P_t)$. Now, we will show that the
  moments of $\mathcal W$ have the desired asymptotic properties, using that $\mathcal W$
  can be represented explicitly via
  \begin{align}
    W_t
     &= \frac{p^{P_t}e^{\delta t}}{X_0^{-1}+\vartheta\int_0^tp^{P_s}e^{\delta s}ds}.\label{eq:W-explicit}
  \end{align}
  Surely, $\mathcal W$ starts in $X_0$ and has the desired jumps at the times of $P$,
  since $W_t/p^{P_t}$ is continuous. To verify the desired deterministic growth, let
  $\eta(t):=X_0^{-1}+\vartheta\int_0^tp^{P_s}e^{\delta s}ds$, the denominator of $W_t$, and
  assume that $P$ does not jump in an interval $(t-\varepsilon,t+\varepsilon)$. Then,
  \begin{align*}
    \tfrac d{dt}W_t
     &= \frac{p^{P_t}\delta e^{\delta t}\eta(t)-p^{P_t}e^{\delta t}\eta'(t)}{\eta(t)^2}
      = W_t\cdot\frac{\delta\eta(t)-\vartheta p^{P_t}e^{\delta t}}{\eta(t)}
      = W_t(\delta-\vartheta W_t)
      = w(W_t).
  \end{align*}
  Next, we consider the process $\overline{\mathcal W}$ that arises from $\mathcal W$
  by exchanging for every $t$ the path $(P_s)_{0\leq s\leq t}$ with its time-reversal
  $(\overline P_s^t)_{0\leq s\leq t}$ via $\overline P_s^t := P_t-P_{t-s}$, given by
  \begin{align*}
    \overline W_t
     &= \frac{p^{\overline P^t_t}e^{\delta t}}{X_0^{-1}+\vartheta\int_0^tp^{\overline P_s^t}e^{\delta s}ds}
      = \frac1{X_0^{-1}p^{-P_t}e^{-\delta t}+\vartheta\int_0^tp^{-P_{t-s}}e^{-\delta(t-s)}ds}\\[1em]
     &= \Big(X_0^{-1}p^{-P_t}e^{-\delta t}+\vartheta\int_0^tp^{-P_s}e^{-\delta s}ds\Big)^{-1}.
  \end{align*}
  Now, since $(P_s)_{0\leq s\leq t}$ is equally distributed as
  $(\overline P_{s+})_{0\leq s\leq t}$ for every $t>0$, where
  $\overline P_{s+} := \lim_{r\downarrow s} \overline P_r$ denotes the
  right-side limit, we also have that $W_t\overset d=\overline W_t$
  for all $t$ and in particular $\E[W_t^k]=\E[\overline W_t^k]$. Also
  note that, in contrast to $W_t$, $\overline W_t$ decreases if the
  path of $(P_t)$ increases pointwise. Now, using this and
  \eqref{lld:ld3} of Lemma \ref{lem:ld-PP}, we obtain for $(L_2)$
  (even for arbitrary $\delta>0$), considering that
  $\delta\leq\hat\alpha$ and thus
  $\gamma\geq\frac1{\hat\alpha}\log\frac1p>1$ as well as
  $\delta=\tfrac1\gamma\log\tfrac1p$,
  \begin{align*}
    \frac{1}{t} \log \E[X_t^k]
     &\geq\frac1t\log\E[\overline W_t^k]\\
     &= \frac1t\log\E\Big[\Big(X_0^{-1}p^{-P_t}e^{-\delta t}
                   + \vartheta\int_0^tp^{-P_s}e^{-\delta s}ds\Big)^{-k}\Big]\\
     &\geq \frac1t \log \E\Big[\Big(X_0^{-1}p^{-P_t}e^{-\delta t}
                   + \vartheta\int_0^tp^{-P_s}e^{-\delta s}ds\Big)^{-k},
                     P_s\leq \tfrac s\gamma\text{ for all }s\leq t\Big]\\
     &\geq \frac 1t \log\Bigg(\Big(X_0^{-1}e^{(\frac1\gamma\log\frac1p-\delta)t}
                   + \vartheta\int_0^t e^{(\frac1\gamma\log\frac1p-\delta)s}ds\Big)^{-k}
                   \mathbb P(P_s \leq \tfrac s\gamma\text{ for all }s\leq t)\Bigg)\\
     &\xrightarrow{t\to\infty}
       -k\cdot0 - (1-\tfrac1\gamma(1+\log\gamma)) = A_\gamma.
  \end{align*}
  For the case $\delta\leq p^k\log\tfrac1p$, we deduce analogously
  \begin{align*}
    \frac{1}{t} \log \E[X_t^k]
     &\geq \frac1t\log\E\Big[\Big(X_0^{-1}p^{-P_t}e^{-\delta t}
                   + \vartheta\int_0^tp^{-P_s}e^{-\delta s}ds\Big)^{-k}\Big]\\
     &\geq \frac 1t \log\Bigg(\Big(X_0^{-1}e^{(p^k\log\frac1p-\delta)t}
                   + \vartheta\int_0^t e^{(p^k\log\frac1p-\delta)s}ds\Big)^{-k}
                   \mathbb P(P_s \leq p^ks\text{ for all }s\leq t)\Bigg)\\
     &\xrightarrow{t\to\infty}
       -k(p^k\log\tfrac1p-\delta) - (1 - p^k + p^k\log(p^k))
      = \delta k - (1 - p^k).
  \end{align*}
  With the same argument as at the end of the proof of $(U_2)$, one can recognise that
  the bound in $(L_1)$ is stronger than $A_\gamma$ if $\delta\leq p^k\log\tfrac1p$.\\[0.5em]
  \emph{2.} Recalling that $Y_t:=-\log X_t$ and letting $T_z:=\inf\{t>0:Y_t=z\}$, we show
  that (i) there is $z\geq 0$ such that $\E[T_z|Y_0=z] < \infty$ and (ii)
  $\mathbb P(T_z<\infty|Y_0 = y) = 1$ for all $y\geq 0$:\\
  Since $\alpha'_0>\log\tfrac1p$, there have to be $\xi>0$ and $\varepsilon>0$ such
  that $\alpha(x)/x \geq (1+\varepsilon)\log\tfrac1p$ for all $x\in(0,\xi]$. Setting
  $z=-\log\xi<\infty$, we see that $\beta(y)\geq(1+\varepsilon)\log\tfrac1p=:\zeta$
  for all $y\geq z$. We define
  \[
    S
     := S_{(z, z+\log(1/p)]}
     := \inf\{t:z<Y_t\leq z+\log(1/p)\}.
  \]
  Then, $\E[S | Y_0 = y] < \infty$ for all $y\leq z$. Indeed, the
  probability for at least $z/\log(1/p)$ jumps in some small time interval
  of length $\varepsilon'>0$ is positive. After the first such time interval
  we can be sure that $S$ has occurred. By finiteness of first moments of
  geometric distributions, $\E[S | Y_0 = y] < \infty$ follows. By a restart
  argument, we have to show that $\E[T_z | Y_0 = y] < \infty$ for all
  $z<y\leq z+\log(1/p)$, which will be done by using a
  comparison argument. For this, let $\mathcal R = (R_t)_{t\geq 0}$ be a
  process with generator
  \[
    \cG_{\mathcal R}g(y)
     = g(y + \log(1/p)) - g(y) -\zeta g'(y).
  \]
  If $z < R_0 = Y_0 \leq z + \log(1/p)$, then -- using the same Poisson
  processes for $\mathcal Y$ and $\mathcal R$ -- we have that
  $T_z \leq T_z^{\mathcal R} := \inf\{t\geq 0:R_t = z\}$ since
  $\beta(y) \geq \zeta$ for $y\geq z$. Analogously to the argument following
  \eqref{eq:bound-Y}, we see that $R_t\to-\infty$ as $t\to\infty$ almost
  surely, which implies that $T_z^{\mathcal R}<\infty$. Since
  $(R_t - R_0 + t(\zeta - \log(1/p)))_{t\geq 0}$ is a martingale, we
  have by optional stopping that
  $
   \E[R_0 - R_{T^{\mathcal R}_z}]
    = R_0 - z
    = (\zeta - \log(1/p)) \E[T_z^{\mathcal R}]
    = \varepsilon \log(1/p) \E[T_z^{\mathcal R}]
  $,
  hence $\E[T_z^{\mathcal R} | Y_0 = R_0] \leq 1/\varepsilon < \infty$.
  It is now straightforward to obtain the properties (i) and (ii).\\
  Now, by (i) and homogeneity, $\cY$ is a positively recurrent delayed
  regenerative process in the sense of Definition 7.1.1 in \cite[p. 387]{BladtNielsen2017}
  with regeneration cycles starting at the state $z$ (and so is $\cX$ with
  cycles starting in $\xi$). By (ii), the delay is almost surely finite.
  Hence, Theorem 7.1.4 in \cite[p. 388]{BladtNielsen2017} gives us weak
  convergence of $\cY$ to a finite random variable $Y_\infty$ and thus,
  also $\cX$ has a weak limit $X_\infty:=e^{-Y_\infty}>0$. Furthermore,
  recalling that $x_\alpha=\min\{x\in(0,1]:\alpha(x)=0\}$, this is well-defined
  since $\alpha$ is positive on $(0,\xi)$ and Lipschitz-continuous on $[\xi,1]$,
  $\alpha'_0>0$ and $\alpha(1)\leq0$. This implies that the hitting time of
  $[0,x_\alpha)$ of $\cX$ is almost surely finite. After that, $\cX$ will never
  hit $[x_\alpha,1]$ again, since it will always jump before it can reach
  $x_\alpha$. Thus, $\Pw(X_\infty\in(0,x_\alpha))=1$.

  \noindent
  If $f:[0,1]\to\R$ is measurable and bounded, Theorem 7.1.6 of
  \cite[p. 391]{BladtNielsen2017} as well as the weak convergence give us that
  almost surely
  \begin{align}
    \lim_{t\to\infty}\frac1t\int_0^tf(X_s)ds
     &= \lim_{t\to\infty}\frac1t\int_0^t\E[f(X_s)]ds
      = \E[f(X_\infty)].\label{eq:ergodic-thm-for-X}
  \end{align}
  Since for $t>0$ and a continuous and bounded function $f$ the distribution
  $\mu$ of $X_\infty$ satisfies
  \begin{align*}
    \E_{\mu}[f(X_t)]
     &= \E\big[\E_{X_\infty}[f(X_t)]\big]
      = \lim_{s\to\infty}\E\big[\E_{X_s}[f(X_t)]\big]
      = \lim_{s\to\infty}\E[f(X_{t+s})]
      = \E[f(X_\infty)],
  \end{align*}
  $\mu$ is a stationary distribution. Letting $\nu$ be a stationary distribution
  of $\cX$ we obtain from $\eqref{eq:ergodic-thm-for-X}$ that for all $t$
  \begin{align*}
    \E_\nu[f(X_1)]
     &= \frac1t\int_0^t\E_\nu[f(X_s)]ds
      = \E[f(X_\infty)],
  \end{align*}
  by which the stationary distribution must be unique. In particular,
  by stationarity $\E[\cG_\cX f(X_\infty)]=0$ holds. Hence, choosing
  $f(x)=\log(\rho+x)$ for $\rho>0$, we obtain
  \begin{align*}
    0
     &= \E\Big[\log\Big(\frac{\rho+pX_\infty}{\rho+X_\infty}\Big)
       + \frac{\alpha(X_\infty)}{\rho+X_\infty}\Big]
      \xrightarrow{\rho\to0} -\log(\tfrac1p) + \E[X_\infty^{-1}\alpha(X_\infty)]
  \end{align*}
  by monotone convergence, since $X_\infty>0$ almost surely.  (Note
  here, that this argument works for both cases $(C_1)$ and $(C_2)$.)
  On the other hand, choosing $f(x)=x^k$, it follows
  \begin{align*} 
    0 
     &= \E[(pX_\infty)^k - X_\infty^k + \alpha(X_\infty)kX_\infty^{k-1}],
  \end{align*}
  which implies
  $\displaystyle
    \E[X_\infty^k]
      = \frac k{1-p^k}\E[X_\infty^{k-1}\alpha(X_\infty)].
  $
\end{proof}

\begin{remark}\label{rem:prop:convergence-pw.det.MPs}
  \ 
  \begin{enumerate}
  \item For $\alpha'_0<\log\tfrac1p<\hat\alpha$, Theorem
    \ref{thm:convergence-$p$-jump-pros} is inconclusive. The weak
    estimates using $\hat\alpha$ in \eqref{eq:bound-Y} and
    \eqref{eq:U1} offer room for improvement. Also note here that
    for the application of $(L_1)$ and $(L_2)$ for $\alpha(0)=0$,
    one can often choose $\delta=\alpha'(0)=\alpha'_0$.
    \item
      Using the notion of \emph{strong large deviations}, e.g.
      \cite[Theorem 3.5, p.1868]{ChagantySethuraman1993}, one can compute exact asymptotics:
      \begin{enumerate}
        \item
          If $x>1$, then
          $\displaystyle{
            \mathbb P(P_t\geq xt)
              \sim \frac{\sqrt x}{x-1}
                \cdot\frac{e^{-t(1-x+x\log x)}}{\sqrt{2\pi t}}}$,
        \item
          If $0<x<1$, then
          $\displaystyle{
            \mathbb P(P_t\leq xt)
              \sim \frac1{(1-x)\sqrt x}
                \cdot\frac{e^{-t(1-x+x\log x)}}{\sqrt{2\pi t}}}$,
      \end{enumerate}
      where $\sim$ denotes asymptotic equivalence, i.e. $f\sim g\Leftrightarrow f(t)/g(t)\to 1$
      as $t\to\infty$.
      Applying these to $u_1$ and $u_2$ in \eqref{eq:u(t)} would provide more precise
      upper bounds in $(U_1)$ and $(U_2)$. For similar stronger bounds in $(L_1)$ and $(L_2)$
      however, one would need strong large deviation results for Poisson processes.
    \item Dropping the boundedness of the state space of $\cX$,
      i.e. considering $s_\alpha=\infty$, \eqref{eq:bound-Y}
      holds letting $\hat\alpha=\sup_{x\in\R_+}\tfrac1x\alpha(x)$ and
      so does \emph{1.} as well as $(U_1)$. Also, $(L_1)$ and $(L_2)$
      still apply, starting $\mathcal W$ at $W_0:=\min\{1,X_0\}$. But
      we need boundedness from below of $\cY$ to prove \eqref{eq:u(t)}
      and the finiteness of $\E[S|Y_0=y]$ and thus lose $(U_2)$ as
      well as the limit results of \emph{2.}
    \item Whenever $(X_t)$ is a $p$-jump process with drift $\alpha$,
      the process of the $k$th powers, $(X_t^k)$, is a $p^k$-jump
      process with drift
      $\alpha_k:x\mapsto kx^{1-\frac1k}\alpha(x^{\frac1k})$ holding
      $\hat\alpha_k=k\hat\alpha$ as well as
      $(\alpha_k)'_0=k\alpha'_0$.  Hence, it would suffice to prove
      $(U_1)$ and $(U_2)$ for $k=1$. For $(L_1)$ and $(L_2)$ however,
      the lower bound on $\alpha$ only implies that
      $\alpha_k(x)\geq k\delta x - k\vartheta x^{1+\frac1k}$, such
      that for $k>1$ we can not use the process $\mathcal W$ in
      \eqref{eq:W-explicit} as a lower bound process and the proof
      fails.
  \end{enumerate}
\end{remark}

\subsection{Proof of Corollary \ref{cor:p-jump-concave}}\label{sec:pr-cor-concave}

\begin{proof}
  The concavity of $\alpha$ implies that $x_\alpha$ equates to $x_\alpha$ and $s_\alpha$
  from Theorem \ref{thm:convergence-$p$-jump-pros}. Noting that, if $I=\R^+$ and
  $s=\max\{x_\alpha,X_0\}$, $\alpha'(s)>-\infty$, by the same scaling argument as in
  the beginning of the proof of Theorem \ref{thm:convergence-$p$-jump-pros} we can
  assume that $I=[0,1]$.\\[.5em]
  We start with $p\in(0,1)$:
  Since $\alpha$ is concave, $x\mapsto\frac{\alpha(x)-\alpha(y)}{x-y}$ decreases on $(y,1]$
  and $y\mapsto\frac{\alpha(x)-\alpha(y)}{x-y}$ decreases on $(0,x)$. Thus, by the assertions,
  $|\frac{\alpha(x)-\alpha(y)}{x-y}|$ is bounded on each interval $[\varepsilon,1]$ and bounded
  on $[0,1]$, if $\alpha'_0<\infty$. Hence, either $(C_1)$ or $(C_2)$ holds, $\alpha$ suffices
  the assertions of Theorem \ref{thm:convergence-$p$-jump-pros} and the existence of $\cX$
  follows and so does \emph3 by Theorem \ref{thm:convergence-$p$-jump-pros}.\emph2.\\
  Recall that $\alpha'_0=\alpha'(0)$, if $\alpha'_0<\infty$. Now, for \emph{1.} we want to
  apply $(L_1)$ and $(L_2)$ of Theorem \ref{thm:convergence-$p$-jump-pros} using
  $\delta=\alpha'_0=\alpha'(0)$. Let us first assume that $\alpha''(0)>-\infty$. Then,
  there is $\vartheta>0$ such that $\alpha'(0)x-\vartheta x^2\leq\alpha(x)$ for all
  $x\in[0,1]$.
  (The Taylor expansion delivers the existence of a $\vartheta_0$ such that
  $\alpha'(0)x-\vartheta_0x^2\leq\alpha(x)$ holds for $x$ in some interval $[0,\varepsilon)$.
  The boundedness of $\alpha$ ensures that we can choose $\vartheta_1$ such that
  $\alpha'(0)x-\vartheta_1x^2\leq\alpha(x)$ holds for $x\geq\varepsilon$. Let
  $\vartheta=\max\{\vartheta_0,\vartheta_1\}$.)\\
  On the other hand, if $\alpha''(0)=-\infty$, $\alpha$ has no parabolic lower bound as
  we need for $(L_1)$ and $(L_2)$ (e.g. for $\alpha(x):=ax - bx^{3/2}$ with $b\geq a>0$).
  We solve this by a coupling argument: Therefore, let $p$ be fixed and for all $n$ define
  \begin{align*}
    \alpha_n(x)
     &= \min\{\alpha(x),x\alpha(\tfrac1n)\},
  \end{align*}
  the minimum of $\alpha$ and the secant of $\alpha$ intersecting at 0 and $\tfrac1n$.
  Surely, the $\alpha_n$ satisfy the conditions of Theorem \ref{thm:convergence-$p$-jump-pros},
  $\alpha_n\nearrow\alpha$ point-wise, $\alpha''(0)=0>-\infty$ and $\alpha_n(0)=0$ such
  that $(\alpha_n)'_0=\alpha_n'(0)$. Let $\cX^{(n)}$ be processes with respective generators
  \begin{align*}
    \cG_nf(x)
     &= f(px)-f(x)+\alpha_n(x)f'(x)
  \end{align*}
  coupled in such a way that they jump simultaneously with $\cX$ and each started in $X_0$.
  Then, $X^{(n)}_t\leq X_t$ for all $n$ and thus
  \begin{align*}
    \liminf_{t\to\infty}\tfrac1t\log\E[X_t^k]
     &\geq \sup_n\liminf_{t\to\infty}\tfrac1t\log\E[(X_t^{(n)})^k].
\intertext{
  From the previous case, we obtain for every $n$
}
    \liminf_{t\to\infty}\tfrac1t\log\E[(X_t^{(n)})^k]
     &\geq
       \begin{cases}
         -(1-p^k-\alpha_n'(0)k) & \text{if }\alpha_n'(0)\leq p^k\log\tfrac1p\\[0.5em]
         -(1-\tfrac1{\gamma_n}(1+\log\gamma_n)) & \text{if }p^k\log\tfrac1p<\alpha_n'(0)<\log\tfrac1p,
       \end{cases}
  \end{align*}
  where $\gamma_n=\log\tfrac1p/\alpha_n'(0)$.  Since
  $\alpha_n'(0)=\alpha(\tfrac1n)/(\tfrac1n)\nearrow\alpha'(0)$ and the
  bounds in either case increase in $\alpha_n'(0)$, we obtain "$\leq$"
  in \emph1. Since $\alpha$ is concave, it satisfies
  $\alpha'(0)=\hat\alpha$ and thus, $(U_1)$ and $(U_2)$ of Theorem
  \ref{thm:convergence-$p$-jump-pros} provide "$\geq$" and \emph{1.} follows.\\
  For the case $\alpha'_0=\log\frac1p$ consider a family
  $(\cX^{(p)})_{p\in(0,1)}$ of processes with respective generators as
  in \eqref{eq:generator-of-p-jump-process} with $\alpha$ fixed and
  $p\mapsto X_0^{(p)}$ constant, coupled in such a way that the
  processes jump simultaneously. Then, for every $t$ the map
  $p\mapsto X_t^{(p)}$ is increasing.  Since
  $p^k\log\frac1p<\alpha'_0$ for $p$ near $p^\ast:=e^{-\alpha'_0}$, we
  conclude from \emph{1.} and the boundedness of the $\cX^{(p)}$ that
  \begin{align*}
    0
     &\geq \lim_{t\to\infty}\tfrac1t\log\E[(X_t^{(p^\ast)})^k]
      \geq \sup_{p<p^\ast}\lim_{t\to\infty}\tfrac1t\log\E[(X_t^{(p)})^k]\\[1em]
     &= \sup_{p<p^\ast}-\Bigg(1-\frac{\alpha'_0}{\log\tfrac1p}
                         \Big(1-\log\Big(\frac{\alpha'_0}{\log\tfrac1p}\Big)\Big)\Bigg)
      = 0.
  \end{align*}
  Since $x\mapsto\frac1x\alpha(x) = \frac1x(\alpha(x)-\alpha(0)) + \frac1x\alpha(0)$
  is decreasing, Theorem \ref{thm:convergence-$p$-jump-pros}.\emph2 implies
  \begin{align*}
    \alpha'_0
     &\geq \E\Big[\liminf_{t\to\infty}\frac{\alpha(X_t^{(p^\ast)})}{X_t^{(p^\ast)}}\Big]
      \geq \sup_{p>p^\ast}\E\Big[\liminf_{t\to\infty}\frac{\alpha(X_t^{(p)})}{X_t^{(p)}}\Big]
      = \sup_{p>p^\ast}\log\tfrac1p
      = \alpha'_0
\intertext{and hence almost surely}
    &\liminf_{t\to\infty}\frac{\alpha(X_t^{(p^\ast)})}{X_t^{(p^\ast)}}
      = \lim_{t\to\infty}\frac{\alpha(X_t^{(p^\ast)})}{X_t^{(p^\ast)}}
      = \alpha'_0.
  \end{align*}
  Since $x\mapsto\alpha(x)/x$ is decreasing, $\limsup_tX^{(p^\ast)}_t$ has to be bounded
  by $m_\alpha:=\sup\{x\in[0,1]:\alpha(x) = x\alpha'_0\}$. Finally, if $\alpha$ is strictly
  concave near 0, $\alpha(x)/x$ strictly decreases near 0 and thus $m_\alpha = 0$, which
  concludes the proof for $p\in(0,1)$.\\[.5em]
  Considering $p=1$, we can ignore the jumps and $\cX$ becomes deterministic, i.e. the
  solution of $\dot X_t=\alpha(X_t)$. Then, it holds if $\alpha'_0<0=\log(\frac1p)$ that
  \begin{align*}
    \frac1t(\log(X_t)-\log(X_0))
     &= \frac1t\int_0^t\frac{\dot X_s}{X_s}ds
      = \frac1t\int_0^t\frac{\alpha(X_s)}{X_s}ds
      \leq \alpha'_0
      < 0.
  \end{align*}
  Hence, $X_t\to0$, $\alpha(X_t)/X_t\to\alpha'_0$ and $\frac1t\log(X_t^k)\to k\alpha'_0$,
  which aligns with \emph1. However, if $\alpha'_0=0$, since $\alpha$ is concave it is
  non-positive. Then, $\cX$ is constant, if $\alpha(X_0)=0$, and it converges monotonically
  to $m_\alpha=\sup\{x\in[0,1]:\alpha(x)=0\}$ if $X_0>m_\alpha$.
  Thus, $\lim_{t\to\infty}\frac1t\log(X_t)=\lim_{t\to\infty}\alpha(X_t)/X_t=0$ giving us
  \emph2. Lastly, in the case of $\alpha'_0>0$, $\cX$ will either grow towards $x_\alpha$
  if started below, i.e. $X_0\in(0,x_\alpha)$, or fall towards it if started above.
  Either way, $X_t\to x_\alpha$, $\alpha(X_t)\to0$ and we obtain \emph3.\\[.5em]
  Choosing $p=0$, $\cX$ would jump to 0 after an exponentially distributed time $T$ with mean
  1 and stay there indefinitely. Thus, we can write $X_t=Y_t\cdot\1_{\{T>t\}}$, where $(Y_t)$
  is the deterministic process arising for $p=1$ discussed above. Then clearly, $\cX$ will
  always converge to 0 almost surely and
  \begin{align*}
    -\frac1t\log(\E[X_t^k])
     &= -\frac 1t\log\Pw(T>t) - \frac kt\log(Y_t)
      \xrightarrow{t\to\infty} 1 + \max\{0,-k\alpha'_0\}.
  \end{align*}
\end{proof}

\section{Branching processes with binomial disasters}\label{sec:bps}

In the following subsections, we borrow ideas from the notion
of duality of Markov processes; see Chapter 4.4 in \cite{EK86}.

\noindent
Recall that two Markov processes $\cZ = (Z_t)_{t\geq 0}$ and
$\mathcal X = (X_t)_{t\geq 0}$ with state spaces $E$ and $E'$ are called
dual with respect to the function $H: E\times E' \to\mathbb R$ if
\begin{align}\label{eq:dual}
  \E[H(Z_t, x) | Z_0=z] = \E[H(z, X_t)|X_0=x]\tag{D}
\end{align}
for all $z\in E, x\in E'$. When one is interested in the process
$\mathcal Z$, this relationship is most helpful if the process
$\mathcal X$ is easier to analyse than the process $\mathcal Z$.
Moreover, frequently, the set of functions $\{H(\cdot,x): x\in E'\}$
is separating on $E$ such that the left hand side of~\eqref{eq:dual}
determines the distribution of $Z_t$. In this case, the distribution
of the simpler process $\mathcal X$ determines via~\eqref{eq:dual} the
distribution of $\mathcal Z$, so analysing $\mathcal Z$ becomes
feasible.

\noindent
There is no straightforward way how to find dual processes, but they
arise frequently in the literature; see \cite{JansenKurt2014} for a
survey. Examples span reflected and absorbed Brownian motion,
interacting particle models such as the voter model and the contact
process, as well as branching processes.

\noindent
A simple way to verify \eqref{eq:dual} for homogeneous $\cZ$
and $\cX$, is to show that
\begin{align}
  \tfrac\partial{\partial t}\E[H(Z_t,x)|Z_0=z]\big|_{t=0}
   &= \tfrac\partial{\partial t}\E[H(z,X_t)|X_0=x]\big|_{t=0}\label{eq:dual-gens}\tag{D'}
\end{align}
for all $z$ and $x$, since then both sides of $\eqref{eq:dual}$ follow the
same evolution.

\subsection{Proof of Theorem \ref{thm:bp-arb.off.dis}}\label{sec:pr-thm2}

In this section we will discuss branching processes of the form of Definition
\ref{def:hom-bp-abr-off-dis}. Hence, let $\cZ:=\cZ^h_{\lambda,q,\kappa,p}$, where
$\lambda\in(0,\infty)$ is the death-rate, $q=(q_k)_{k\geq0}$ the offspring distribution
on $\N_0$, $p\in(0,1)$ the survival probability of the disasters that occur at the
jump times of $(D_t)_{t\geq0}$, a Poisson process with rate $\kappa>0$. Moreover, let
$h:[0,1]\to[0,1],x\mapsto\sum_{k\geq0}q_kx^k$ be the probability generating function
of the offspring distribution. We start with establishing a suitable duality for $\kappa=1$.
The general case will follow by a rescaling argument.

\begin{lemma}\label{lem:arb-off-dualilty}
  Let $p\in[0,1]$, $\kappa=1$ and $(X_t)$ be a $p$-jump process with drift
  $x\mapsto\lambda(1-x-h(1-x))$, having the generator
  \begin{align}
    \mathcal G_{\mathcal X}f(x)
     &= f(px)-f(x) + \lambda\big(1-x-h(1-x)\big)f'(x)
     \label{eq:generator-dual-arb.off.dis}
  \end{align}
  for $f\in\mathcal C^1_b([0,1])$. Then, the duality relation
  \begin{align}
    \E[(1-X_t)^z|X_0=x]
     &= \E[(1-x)^{Z_t}|Z_0=z]
     \label{eq:duality-arb.off.dis}
  \end{align}
  holds for every $x\in[0,1],z\in\N_0$ and $t\geq0$.
\end{lemma}
\begin{proof}
  Recalling the generator of $\cZ$ from Definition \ref{def:hom-bp-abr-off-dis}, we
  obtain for $x\in[0,1],z\in\mathbb N_0$ and $H(x,z):=(1-x)^z$ that
  \begin{align*}
    \big(\mathcal G_{\mathcal Z} & H(x,\,\cdot\,)\big)(z)\\
     &= \lambda z(1-x)^{z-1}\sum_{k\geq0}q_k\big((1-x)^k-(1-x)\big)
       +\sum_{\ell=0}^z\genfrac(){0pt}{}z\ell\big(p(1-x)\big)^\ell(1-p)^{z-\ell} - (1-x)^z\\[1em]
     &= -\frac{\partial}{\partial x}(1-x)^z\cdot\lambda\big(h(1-x)-(1-x)\big)
        +\big(p(1-x)+(1-p)\big)^z - (1-x)^z\\[1em]
     &= \big(\mathcal G_{\mathcal X}H(\,\cdot\,,z)\big)(x),
  \end{align*}
  which resembles \eqref{eq:dual-gens}. Hence, \eqref{eq:dual} gives us the desired
  relation.
\end{proof}

\noindent
Now, we apply Corollary \ref{cor:p-jump-concave} to the dual process $\cX$, followed by the proof of
Theorem \ref{thm:bp-arb.off.dis}.

\begin{lemma}\label{lem:arb-off-appThm1}
  For $p\in[0,1]$, such a process $\cX$ in Lemma \ref{lem:arb-off-dualilty} exists
  and satisfies
  \begin{enumerate}
  \item if $p=0$ or $h'(1) < 1+\tfrac1\lambda\log\tfrac1p$,
    $X_t\xrightarrow{t\to\infty}0$ almost surely.  Also, for $k\geq1$
    \begin{align*}
      \lim_{t\to\infty}-\tfrac1t\log\E[X_t^k]
      &= \begin{cases}
        1+\max\{0,-k(h'(1)-1)\} & \text{if }p=0,\\
        1-p^k - k\lambda(h'(1)-1)) & \text{if }h'(1)\leq 1+\tfrac{p^k}\lambda\log\tfrac1p,\\
        1-\tfrac1\gamma(1+\log\gamma) & \text{otherwise,}
      \end{cases}
    \end{align*}
    where $\gamma=\tfrac1\lambda\log\tfrac1p/(h'(1)-1)$.
  \item if $h'(1)=1+\tfrac1\lambda\log\tfrac1p$,
    $X_t\xrightarrow{t\to\infty}0$ almost surely, while
    $\frac1t\log\E[X_t^k]\xrightarrow{t\to\infty}0$ for all $k$.
  \item if $h'(1)\in(1+\tfrac1\lambda\log\tfrac1p,\infty]$, letting
    $x_\ast$ be the smallest fixed point of $h$, $\cX$ converges weakly
    to a random variable $X_\infty$ on $(0,1-x_\ast]$ that satisfies
    $\E[X_\infty^{-1}(1-h(1-X_\infty))]=1+\tfrac1\lambda\log\tfrac1p$
    and for $k\geq1$
    \begin{align}
      \E[X_\infty^k]
        &= \frac{\lambda k}{\lambda k + 1-p^k}\E[X_\infty^{k-1}(1-h(1-X_\infty))].
         \label{eq:recursion-dual-moments-arb-off}
    \end{align}
  \end{enumerate}
\end{lemma}
\begin{proof}
  Since $h$ is a convex function, $\alpha:x\mapsto\lambda(1-x-h(1-x))$
  is concave.  Also, $\alpha(0)=\lambda(1-h(1))=0$,
  $\alpha(1)=-\lambda q_0\leq0$,
  $\alpha'_0=\alpha'(0)=\lambda(h'(1)-1)\in(-\lambda,\infty]$ and
  $\alpha'(1)=\lambda(h'(0)-1)=\lambda(q_1-1)\geq-\lambda>-\infty$. Hence,
  considering that $\alpha'(0)\geq p^k\log\tfrac1p$ iff
  $h'(1)\leq 1+\tfrac{p^k}\lambda\log\tfrac1p$ for all $k\geq1$,
  Corollary \ref{cor:p-jump-concave} implies \emph{1}. For \emph{3.},
  noting that $\alpha(x)>0$ only if $0<x<1-x_\ast=x_\alpha$, only
  \eqref{eq:recursion-dual-moments-arb-off} remains to be shown.
  Here, Corollary \ref{cor:p-jump-concave}.\emph3 gives us for $k\geq1$
  \begin{align*}
    \E[X_\infty^k]
     &= \frac{\lambda k}{1-p^k}\big(-\E[X_\infty^k]+\E[X_\infty^{k-1}(1-h(1-X_\infty))]\big)
  \end{align*}
  and \eqref{eq:recursion-dual-moments-arb-off} follows. Finally, if
  $h'(1)=1+\frac1\lambda\log\frac1p>1$, $h$ is strictly convex and thus,
  $\alpha$ is strictly concave, which gives us \emph{2.}
\end{proof}

\begin{proof}[Proof of Theorem \ref{thm:bp-arb.off.dis}]
  First, let the theorem hold for $\kappa=1$ and for arbitrary $\kappa>0$
  consider the process $\cZ^\ast:=\cZ^h_{\lambda/\kappa,q,1,p}$. Then, 
  $(Z_t)_t:=(Z_{\kappa t}^\ast)_t$ defines a $\cZ^h_{\lambda,q,\kappa,p}$-process
  and we obtain
  \begin{align*}
    \lim_{t\to\infty}-\tfrac1t\log\Pw(Z_t>0)
     &= \kappa\cdot\lim_{s\to\infty}-\tfrac1s\log\Pw(Z_s^\ast>0),
  \end{align*}
  which shows \emph{1.} and \emph2. Since $\lim_tZ_t=\lim_tZ_t^\ast$ almost
  surely, \emph{3.} follows, where we obtain \eqref{eq:bp-arb.off.dis-surv.prob.rec}
  by substitution.\\
  Hence, without loss of generality let $\kappa=1$. Then, letting $Z_0=k$ and $X_0=1$,
  Lemma \ref{lem:arb-off-dualilty} implies,
  \begin{align}
    \mathbb P(Z_t>0)
     &= 1-\E[0^{Z_t}]
      = 1-\E[(1-X_t)^k]
      = \sum_{\ell=1}^k\genfrac(){0pt}{}k\ell(-1)^{\ell+1}\E[X_t^\ell].\label{eq:prThm2-1}
  \end{align}
  Considering that $X_t\in[0,1]$, we obtain from Bernoulli's inequality
  \begin{align}
    \E[X_t]
     &\leq 1-\E[(1-X_t)^k]
      = \mathbb P(Z_t>0)
      \leq k\E[X_t].\label{eq:prThm2-2}
  \end{align}
  Thus, noting that $h'(1)=\mu$ and $\gamma=1/\nu$, for $\mu\leq1+\tfrac1\lambda\log\tfrac1p$
  (i.e. $\nu\leq1$) Lemma \ref{lem:arb-off-appThm1}.\emph1 and \ref{lem:arb-off-appThm1}.\emph2
  show that
  \begin{align*}
    \lim_{t\to\infty}-\tfrac1t\log\mathbb P(Z_t>0)
     &= \begin{cases}
          1+\max\{0,-k(h'(1)-1)\} & \text{if }p=0,\\
          1-p-\lambda(\mu-1) & \text{if }\nu\leq p,\\
          1-\nu(1+\log\tfrac1\nu) & \text{if }p<\nu\leq1.
        \end{cases}
  \end{align*}
  Additionally, considering the boundedness and thus the $\mathcal L^1$-convergence of $(X_t)_t$,
  we get from \eqref{eq:prThm2-2} that $Z_t$ converges to 0 in probability. Since
  this implies almost sure convergence of a subsequence and 0 is an absorbing state, we have
  $Z_t\to0$ almost surely.\\[0.5em]
  For \emph{2.}, noting that
  $\mathbb P(Z_t\xrightarrow{t\to\infty}0\mid Z_s)\geq\frac1{1+\lambda}(1-p)^{Z_s}$,
  i.e. the probability that the next event after $s$ is a disaster that kills all,
  we obtain
  \begin{align*}
    \limsup_{s\to\infty}\mathbb P(Z_t\xrightarrow{t\to\infty}0\mid\sigma(Z_r;r\leq s))
     \geq \tfrac1{1+\lambda}\limsup_{s\to\infty}(1-p)^{Z_s}
     = \tfrac1{1+\lambda}(1-p)^{\liminf\limits_{s\to\infty}Z_s}.
  \end{align*}
  Thus, Lemma 3.1 of \cite[p.54]{KaplanEtAl1975} concludes
  \begin{align*}
    \mathbb P(Z_t\xrightarrow{t\to\infty}0)
     + \mathbb P(Z_t\xrightarrow{t\to\infty}\infty)
     = 1.
  \end{align*}
  Furthermore, for $\mu>1+\tfrac1\lambda\log\tfrac1p$ Lemma \ref{lem:arb-off-appThm1}.\emph3
  shows stationarity of the distribution of $X_\infty$ and hence independence of $X_0$. Thus,
  using that 0 is an absorbing state and $\{Z_s=0\}\subset\{Z_t=0\}$ for $s\leq t$, we obtain
  from \eqref{eq:prThm2-1} that
  \begin{align*}
    \mathbb P(\lim_{t\to\infty}Z_t=0)
     &= \mathbb P\Big(\bigcup_{t>0}\{Z_t=0\}\Big)
      = \lim_{t\to\infty} \mathbb P(Z_t=0)
      = \E[(1-X_\infty)^k].
  \end{align*}
\end{proof}

\subsection{Preparation: Regular Variation}\label{sec:reg-var}

In this subsection, using results of chapter VIII.9 of \cite{Feller1971} and
\cite{Seneta1976}, we will arrange the tools regarding regularly varying functions
needed for the proof of Theorem \ref{thm:inhom-bdp}. However, we need to establish
some additional notation first:

\begin{remark}\label{rem:notation}\ 
 \begin{enumerate}
  \item
    We will make use of the Bachmann-Landau notation: For a function
    $g: \R_+\to[0,\infty)$, let
    \begin{align*}
      o(g) &:= \{f: \R_+\to \R_+\mid \limsup_{t\to\infty}\tfrac{f(t)}{g(t)}=0\},\\
      O(g) &:= \{f: \R_+\to \R_+\mid \limsup_{t\to\infty}\tfrac{f(t)}{g(t)}<\infty\},\\
      \Omega(g)
           &:= \{f:\R_+\to \R_+\mid g\in O(f)\}.
    \end{align*}
  \item
    We define the relation of \emph{asymptotic equivalence} for functions
    $f,g:\R_+\to\R$ by
    \begin{align*}
      f\overset{t\to\infty}\sim g
        &\quad \Leftrightarrow \quad \lim_{t\to\infty}\frac{f(t)}{g(t)}=1.
    \end{align*}
    Often, when the running variable is either obvious or $t$, we will just write $f\sim g$.
 \end{enumerate}
\end{remark}

\begin{definition}\label{def:reg-var}
  A function $f:\R_+\to \R_+$ is called \emph{regularly varying with
    exponent $\beta\in\R$}, if for every $x>0$
  \begin{align*}
    \frac{f(xt)}{f(t)}
      \xrightarrow{t\to\infty} x^\beta
  \end{align*}
  holds. A \emph{slowly varying} function is a regularly varying
  function with exponent 0.
\end{definition}

\begin{lemma}\label{lem:reg-var}
  Let $f:\R_+\to \R_+$ regularly varying with exponent $\beta\in\R$
  and $F(t):=\int_0^tf(x)dx$.
  \begin{enumerate}
    \item
      $F$ is regularly varying with exponent $\max\{\beta+1,0\}$ and for $t\to\infty$, if
      \begin{enumerate}
        \item $\beta>-1$, then $F(t) \sim tf(t)(\beta+1)^{-1}$.
        \item $\beta=-1$, then $F(t) \in \Omega(1)\cap O(t^\varepsilon)$ for all $\varepsilon>0$.
        \item $\beta<-1$, then $F(t) \rightarrow c <\infty$.
      \end{enumerate}
    \item
      Let $(t_n)\subset \R_+$ such that $t_n\xrightarrow{n\to\infty}\infty$
      and $t_{n+1}/t_n\xrightarrow{n\to\infty}1$. Then,
      \begin{align*}
        \frac{F(t_{n+1})}{F(t_n)}
          \xrightarrow{n\to\infty} 1.
      \end{align*}
    \item
      There are functions $a$ and $\varepsilon$ such that $a(t)\xrightarrow{t\to\infty}c\in\R_+$,
      $\varepsilon(t)\xrightarrow{t\to\infty}0$ and
      \begin{align*}
        f(t)
         &= t^\beta a(t)\exp\Big(\int_1^t\frac{\varepsilon(y)}ydy\Big).
      \end{align*}
    \item
      For each $\alpha>0$, $f\in O(t^{\beta+\alpha})\cap\Omega(t^{\beta-\alpha})$.
  \end{enumerate}
\end{lemma}
\begin{proof}
  \emph3. and \emph4. follow from \cite{Seneta1976}, Theorem 1.1 on page 2 and
  Proposition $1^0$ on page 18 respectively, while \emph1. is a consequence of
  \emph4. and exercises 2.1, 2.2 and 2.3 on \cite[p.86]{Seneta1976}.
  (A proof of these exercises is given by Theorem 1 in \citealp[p.281]{Feller1971}.)\\
  Finally, by \emph{1.} $F$ is regularly varying with exponent $\beta'\geq0$.
  Applying \emph{3.} we see that there are functions $A$ and $\mathcal E$ with
  $\lim_{t\to\infty}A(t)=c\in(0,\infty)$ and $\lim_{t\to\infty}\mathcal E(t)=0$
  such that
  \begin{align*}
    \frac{F(t_{n+1})}{F(t_n)}
     &= \Big(\frac{t_{n+1}}{t_n}\Big)^{\beta'}
          \cdot\frac{A(t_{n+1})}{A(t_n)}
          \cdot\exp\Big(\int_{t_n}^{t_{n+1}}\frac{\mathcal E(y)}ydy\Big).
  \end{align*}
  Now, the first two factors converge to 1, while the integral in the exponent is bounded by
  $|t_{n+1}-t_n|\cdot\frac1{t_n}\sup_{y\in[t_n,t_{n+1}]}|\mathcal E(y)|\longrightarrow_{n\to\infty}0$.
\end{proof}

\noindent
The following Theorem is needed in the proof of Theorem \ref{thm:inhom-bdp} to build a
bridge between the asymptotics of the deterministic rate functions and the almost
sure asymptotics of the process $(L_t)$ from Lemma \ref{lem:inhom-duality}, which is
key to the computation of the survival probability in the inhomogeneous case.

\begin{theorem*}\label{thm:asymptotics-of-D-integral}
  Let $(D_t)_{t\geq0}$ be an inhomogeneous Poisson process with right
  continuous rate-function $\kappa$ with left limits,
  $\Lambda(t):=\int_0^t\kappa_sds$,
  $\Lambda^{-1}(t):=\inf\{s>0:\Lambda(s)>t\}$ and $f: \R_+\to \R_+$
  such that $f(\Lambda^{-1}(\cdot))$ is regularly varying with
  exponent $\beta$.
  \begin{enumerate}
    \item
      If $\Lambda(t)\xrightarrow{t\to\infty}\Lambda(\infty)<\infty$ or
      $\beta<-1$, then $\int_0^tf(s)dD_s$ has an almost surely finite limit.
    \item
      If $\Lambda(t)\xrightarrow{t\to\infty}\infty$ and $\beta>-1$,
      \begin{align*}
        \int_0^tf(s)dD_s
          \sim \int_0^tf(s)\kappa_sds
      \end{align*}
      holds almost surely and in $\mathcal L^2$.
    \item
      If $\beta=-1$, for arbitrary $\alpha>0$ it holds
      $\displaystyle
        \frac{\int_0^tf(s)dD_s}{\int_0^tf(s)\kappa_sds}
          \in O(t^\alpha)\cap\Omega(t^{-\alpha})
      $
      almost surely.
  \end{enumerate}
\end{theorem*}
\begin{proof}
  First note that there is a unit-rate Poisson process, which we denote by $(P_t)$
  and its jump times by $(\sigma_k)_k$, such that $D_t=P_{\Lambda(t)}$ for all
  $t\geq0$ and the jump times of $(D_t)$ satisfy $\tau_k=\Lambda^{-1}(\sigma_k)$.
  Then, supposing that \emph{2.} holds for $\kappa\equiv1$, the general case follows as
  \begin{align*}
    \int_0^tf(s)dD_s
     &= \sum_{k=1}^{D_t}f(\tau_k)
      = \sum_{k=1}^{P_{\Lambda(t)}}f(\Lambda^{-1}(\sigma_k))
      \sim
        \int_0^{\Lambda(t)}f(\Lambda^{-1}(s))ds
      = \int_0^tf(s)\kappa_sds.
  \end{align*}
  Thus, without loss of generality, let $\kappa\equiv1$, $(D_t)=(P_t)$ and $f$
  regularly varying with exponent $\beta>-1$. Letting $F(t)=\int_0^tf(x)dx$, it
  remains to be shown that
  \begin{align*}
    Y_t
     &:= \frac1{F(t)}\sum_{k=1}^{D_t}f(\tau_k)
      \xrightarrow{t\to\infty}1
  \end{align*}
  almost surely and in $\mathcal L^2$. Starting with the $\mathcal L^2$-convergence,
  we recall that on the event $\{D_t=n\}$, the jump times $(\tau_1,\ldots,\tau_n)$ are
  equal in distribution to $(U^t_{(1)},U^t_{(2)},\ldots,U^t_{(n)})$, the order
  statistic of $n$ iid variables $(U^t_1,\ldots,U^t_n)$, uniformly distributed on
  $[0,t]$. We obtain
  \begin{align}
    \E\Big[\sum_{k=1}^{D_t}f(\tau_k)\Big]
     &= \E\Bigg[\sum_{k=1}^{D_t}\E\Big[f(\tau_k)\Big|D_t\Big]\Bigg]
      = \E\Bigg[\sum_{k=1}^{D_t}\E\Big[f(U^t_k)\Big|D_t\Big]\Bigg]\notag \\[1em]
     &= t\E[f(U^t_1)]
      = t\cdot\frac1t\int_0^tf(s)ds
      = F(t).
      \label{eq:expectation-D-integral}
  \end{align}
  Hence, $\E[Y_t]=1$ and we compute
  \begin{align*}
    \|Y_t-1\|_{\mathcal L^2}
     &= \text{Var}\Bigg[\frac1{F(t)}\sum_{k=1}^{D_t}f(\tau_k)\Bigg]
      = \frac1{F(t)^2}\Bigg(\text{Var}\Bigg[\sum_{k=1}^{D_t}\E\Big[f(U^t_k)\Big|D_t\Big]\Bigg]
          + \E\Bigg[\sum_{k=1}^{D_t}\text{Var}\Big[f(U^t_k)\Big|D_t\Big]\Bigg]\Bigg)\\[1em]
     &= \frac1{F(t)^2}\Big(\text{Var}\Big[D_t\E[f(U^t_1)]\Big]
          + \E\Big[D_t\text{Var}[f(U^t_1)]\Big]\Big)
      = \frac1{F(t)^2}\Big(t\E[f(U^t_1)]^2 + t\text{Var}[f(U^t_1)]\Big)\\[1em]
     &= \frac{t\E[f(U^t_1)^2]}{F(t)^2}
      = \frac{\int_0^tf(x)^2dx}{\big(\int_0^tf(x)dx\big)^2}.
  \end{align*}
  Since $f^2$ is regularly varying with exponent $2\beta$, we obtain from
  Lemma \ref{lem:reg-var}.\emph1 and \ref{lem:reg-var}.\emph4 for
  \begin{itemize}
    \item
      $\beta>-\frac12$, that $\|Y_t-1\|_{\mathcal L^2}\sim\frac1t\cdot\frac{(\beta+1)^2}{2\beta+1}$.
    \item
      $\beta=-\frac12$, some slowly varying function $\ell$ and arbitrary $\varepsilon>0$ that
      \begin{align*}
        \|Y_t-1\|_{\mathcal L^2}
          = \frac{\ell(t)}t
          \in O(t^{-1+\varepsilon}).
      \end{align*}
    \item
      $\beta\in(-1,-\frac12)$, the numerator converges to a constant and the denominator
      converges to $\infty$.
  \end{itemize}
  Either way, the $\mathcal L^2$ convergence follows.\\
  For the almost sure convergence first note that there is a
  subsequence $(t_n)_n$ with $t_n\nearrow\infty$ as well as
  $\lim_{n\to\infty}Y_{t_n}=1$ and hence,
  $\liminf_tY_t\leq1\leq\limsup_tY_t$ almost surely. Noting that
  $(Y_t)$ is a piecewise deterministic process, jumping upwards and
  between jumps decreasing continuously, the maximum and minimum on
  the $n$th deterministic piece of the path respectively are given by
  \begin{align*}
    Y^+_n
      := Y_{\tau_n}
       = \frac1{F(\tau_n)}\sum\limits_{k=1}^nf(\tau_k)\quad\text{and}\quad
    Y^-_n
      := Y_{\tau_{n+1}-}
       = \frac1{F(\tau_{n+1})}\sum\limits_{k=1}^nf(\tau_k)
  \end{align*}
  and we obtain for every $t$ that $Y^-_{D_t}\leq Y_t\leq Y^+_{D_t}$. Also, we deduce
  from Lemma \ref{lem:reg-var}.\emph2 that
  \begin{align*}
    \frac{Y^+_n}{Y^-_n}
     &= \frac{F(\tau_{n+1})}{F(\tau_n)}
      \xrightarrow{n\to\infty}1
  \end{align*}
  almost surely, considering that $\tau_{n+1}/\tau_n\xrightarrow{n\to\infty}1$. Since
  the values of the local extrema of $(Y_t)_t$ are given by $Y^+$ and $Y^-$, it follows
  that
  \begin{align*}
    \liminf_{n\to\infty}Y^+_n
     &= \liminf_{n\to\infty}Y^-_n
      = \liminf_{t\to\infty}Y_t
      \leq 1
      \leq \limsup_{t\to\infty}Y_t
      = \limsup_{n\to\infty}Y^+_n.
  \end{align*}
  Hence, it suffices to show that $Y^+_n\xrightarrow{n\to\infty}1$ almost surely. For this,
  let $h(n):=\min\{\sqrt n,\sqrt{F(n)}\}$ and decompose $Y^+_n$ in the following way:
  \begin{align}
    Y^+_n
     &= \frac1{F(\tau_n)}\sum_{k\leq h(n)}f(\tau_k)
       + \frac1{F(\tau_n)}\sum_{h(n)<k\leq n}\frac1{f(\tau_k)}.
       \label{eq:Yn+decomposition}
  \end{align}
  Considering that $A^-:=\inf_{n\geq1}\frac{\tau_n}n>0$ and $A^+:=\sup_{n\geq1}\frac{\tau_n}n<\infty$
  almost surely, by Lemma \ref{lem:reg-var} we obtain for the first part
  \begin{align*}
    \frac1{F(\tau_n)}\sum_{k\leq h(n)}f(\tau_k)
    \ &\sim\ \frac{\beta+1}n\sum_{k\leq h(n)}\Big(\frac{\tau_k}{\tau_n}\Big)^\beta
                \cdot\frac{a(\tau_k)}{a(\tau_n)}
                \cdot\exp\Big(\int_{\tau_k}^{\tau_n}\frac{\varepsilon(y)}ydy\Big)\\[1em]
      &\leq \frac Cn\cdot\Big(\frac{A^+}{A^-}\Big)^{|\beta|}
                \exp\Big(\int_0^{nA^+}\frac{|\varepsilon(y)|}ydy\Big)
                \cdot\sum_{k\leq h(n)}\Big(\frac kn\Big)^\beta,
  \end{align*}
  where the constant $C$ arises from the boundedness of $a$. Now, for $\beta\geq0$ the
  remaining sum is bounded above by $h(n)\leq\sqrt n$, while for $-1<\beta<0$ it holds
  for some slowly varying function $\ell$ that
  \[
    \sum_{k\leq h(n)}\Big(\frac kn\Big)^\beta
      \leq \sum_{k\leq h(n)}n^{|\beta|}\leq h(n)n^{|\beta|}
      = n^{|\beta|+\frac{1+\beta}2}\ell(n)
      = n^{\frac{1+|\beta|}2}\ell(n).
  \]
  Thus, noting that $\frac{1+|\beta|}2<1$ and $\exp(\int_0^\bullet\frac{|\varepsilon(y)|}ydy)$
  is slowly varying, it follows that
  \[
    \frac1{F(\tau_n)}\sum_{k\leq h(n)}f(\tau_k)
      \xrightarrow{n\to\infty}0
  \]
  almost surely. Hence, since $\tau_n\sim n$ almost surely and thus $F(n)\sim F(\tau_n)$
  and $f(\tau_k)\sim f(k)$ by Lemma \ref{lem:reg-var}.\emph2, it follows from
  \eqref{eq:Yn+decomposition} and Lemma \ref{lem:reg-var}.\emph{1(a)}
  \begin{align*}
    Y^+_n
    \ &\sim\ 
    \frac{F(n)}{F(\tau_n)}\cdot\frac1{F(n)}\sum_{h(n)<k\leq n}\frac{f(\tau_k)}{f(k)}\cdot f(k)
     \ \sim\ \frac{\beta+1}n\sum_{h(n)<k\leq n}\frac{f(k)}{f(n)}\\[1em]
    \ &\sim\ \frac{\beta+1}n\sum_{h(n)<k\leq n}\Big(\frac kn\Big)^\beta
     \ \sim\ (\beta+1)\int_{h(n)/n}^1x^\beta dx
      \xrightarrow{n\to\infty}1
  \end{align*}
  and the proof of \emph{2.} is done.\\[0.5em]
  For \emph{1.} if $\Lambda(\infty)<\infty$, also
  $
    \lim_{t\to\infty}\int_0^tf(s)dD_s
      = \sum_{k=1}^{P_{\Lambda(\infty)}}f(\tau_k)
  $ is almost surely finite. Otherwise, we obtain from \eqref{eq:expectation-D-integral}
  and Lemma \ref{lem:reg-var}.\emph1 that
  \begin{align*}
    \E\Big[\int_0^tf(s)dD_s\Big]
     &= \int_0^tf(s)\kappa_sds
      \leq \int_0^\infty f(s)\kappa_sds
      = \int_0^\infty f(\Lambda^{-1}(s))ds
      < \infty,
  \end{align*}
  which also, by monotone convergence, implies the finiteness of $\int_0^\infty f(s)dD_s$
  and \emph{1.} is done.\\[.5em]
  Lastly, for \emph{3.} we conclude that for $\alpha>0$, $F(t):=\int_0^tf(s)\kappa_sds$
  and $Y_t:=\int_0^tf(s)dD_s/F(t)$
  \begin{align*}
    0
     &\leq t^{-\alpha}Y_t
      \leq \frac1{F(t)}\int_0^ts^{-\alpha}f(s)dD_s.
  \end{align*}
  Now, since $t\mapsto t^{-\alpha}f(t)$ is regularly varying with exponent $-1-\alpha<-1$,
  \emph{1.} shows that the integral almost surely converges to some finite limit and hence,
  considering that $F$ is non-decreasing and non-negative, $\limsup_t t^{-\alpha}Y_t<\infty$
  almost surely and $Y_t\in O(t^\alpha)$. Similarly, using \emph{2.}
  \begin{align*}
    t^\alpha Y_t
     &\geq \frac1{F(t)}\int_0^ts^\alpha f(s)dD_s,
  \end{align*}
  which either converges to a positive constant, if $\Lambda(\infty)<\infty$, or is
  asymptotically equivalent to
  \begin{align*}
    \frac{\int_0^ts^\alpha f(s)\kappa_sds}{\int_0^tf(s)\kappa_sds}
     &\sim \frac{\int_1^ts^\alpha f(s)\kappa_sds}{\int_1^tf(s)\kappa_sds}
      \geq 1.
  \end{align*}
  Either way, it follows that $\limsup_t t^\alpha Y_t>0$ and thus $Y_t\in\Omega(t^{-\alpha})$.
\end{proof}

\begin{remark}[More precise asymptotics for $\beta=-1$]
  In the case $\beta=-1$ it follows from Lemma \ref{lem:reg-var} that
  $F$ is slowly varying and thus, considering its monotonicity, lies
  in $O(t^\varepsilon)\cap\Omega(1)$ for all $\varepsilon>0$. As
  discussed in \cite{Polfeldt1969} however, it is not always the case,
  that a regularly varying function with exponent $-1$ is integrable
  on $\R_+$. Supposing that $F(\infty)=\infty$, we obtain the
  $\mathcal L^2$-convergence in Theorem \ref{thm:asymptotics-of-D-integral}
  analogously to the case $\beta\in(-1,-\frac12)$, while the methods we
  used to obtain almost sure convergence fail for $\beta=-1$. Conversely,
  if $F(\infty)<\infty$, similarly to the proof of Theorem
  \ref{thm:asymptotics-of-D-integral}.\emph1 it follows that
  \begin{align*}
    \lim_{t\to\infty}\E\Big[\int_0^tf(s)dD_s\Big]
      < \infty
  \end{align*}
  and the integral has a finite almost sure limit. Surely, \cite{Polfeldt1969}
  can be used to specify the results for this critical case.
\end{remark}

\subsection{Proof of Theorem \ref{thm:inhom-bdp}}\label{sec:pr-thm3}

In this section we generalise the findings of Corollary \ref{cor:results-homog-bd-proc} to
the time-inhomogeneous case. Recalling Definition \ref{def:inhom-bd-w-dis}, let
$\cZ=\cZ^{in}_{b,d,\kappa,p}$ with birth, death and disaster rate functions $b$, $d$ and
$\kappa$ respectively, and $p: \R_+\to[0,1]$ the survival probability function.
Furthermore, let $(D_t)_t$ be the inhomogeneous Poisson process with rate $\kappa$ that
counts the disasters up to time $t$.
In what follows we will always assume $b,d$ and $\kappa$ to be right continuous with
left limits and $p$ to be left continuous with right limits.\\[0.5em]
We start by computing the pgf of $\mathcal Z$ for $(1-p)\kappa\equiv0$, i.e. without disasters,
which will be generalised in Lemma \ref{lem:inhom-duality}.

\begin{lemma}\label{lem:inhom-dual-no.dis.}
  Let $v(t):=\int_0^t(b_y-d_y)dy$ and $(1-p)\kappa\equiv0$. Then, for $x\in[0,1]$,
  $t\geq t_0\geq0$ and $k\geq0$ it holds that
  \begin{align*}
    \E[(1-x)^{Z_t}|Z_{t_0}=k]
     &= (1-s(t,x))^k,
  \end{align*}
  where $s(t,0)=0$ for all $t$ and
  $\displaystyle
    s(t,x)^{-1}
      = \tfrac1xe^{v(t_0)-v(t)} + e^{v(t_0)}\int_{t_0}^tb_ye^{-v(y)}dy
  $ for $x>0$.
\end{lemma}
\begin{proof}
  Given that $Z_{t_0}=k$, $Z_t$ is equal in distribution to a sum of $k$ independent copies
  started in 1 at time $t_0$. Thus, $\E[(1-x)^{Z_t}|Z_{t_0}=k]=\E[(1-x)^{Z_t}|Z_{t_0}=1]^k$.
  Hence, without loss of generality we assume $k=1$.
  Now, considering \cite{Kendall1948}, where birth- and death-rate are denoted by $\lambda$
  and $\mu$ respectively and $v$ is denoted by $-\rho$ (cf. (11)), by $(9),(12)$ and $(10b)$
  we can compute for $z\in[0,1]$ that
  \begin{align*}
    \E[z^{Z_t}|Z_0=1]
     &=: \varphi(z,t)
      = \frac{1+e^{v(t)}\int_0^tb_se^{-v(s)}ds-e^{v(t)} + \Big(e^{v(t)}-e^{v(t)}\int_0^tb_se^{-v(s)}ds\Big)z}
             {1 + \int_0^tb_se^{-v(s)}ds - \int_0^tb_se^{-v(s)}ds\cdot z}\\[1em]
     &= 1 - \frac{e^{v(t)} - e^{v(t)}z}{1 + \int_0^tb_se^{-v(s)}ds\cdot(1-z)}.
  \end{align*}
  Substitution of $z=1-x$ and reducing the fraction by $xe^{v(t)}$ concludes the proof for
  $t_0=0$.
  The general case $t_0\geq0$ is obtained considering a process $\cZ^\ast$ with birth and
  death rates at time $s$ given by $b^\ast(s):=b_{t_0+s}$ and $d^\ast:=d_{t_0+s}$
  respectively. Then, for $t\geq t_0$
  \begin{align}
    \E[(1-x)^{Z_t}|Z_{t_0}=1]
     &= \E[(1-x)^{Z^\ast_{t-t_0}}|Z^\ast_0=1]
      = 1 - \frac1{\frac1xe^{-v^\ast_{t-t_0}} + \int_0^{t-t_0}b^\ast(s)e^{-v^\ast_s}ds},
      \label{eq:lem:inhom-no.dis.1}
\intertext{where}
    v^\ast_s
     &= \int_0^s\big(b^\ast(y)-d^\ast(y)\big)dy
      = \int_{t_0}^{t_0+s}\big(b_y-d_y\big)dy
      = v(t_0+s)-v(t_0).
      \label{eq:lem:inhom-no.dis.2}
  \end{align}
  Substituting $y:=s+t_0$ in \eqref{eq:lem:inhom-no.dis.1} and using \eqref{eq:lem:inhom-no.dis.2}
  concludes the proof.
\end{proof}

\noindent
The following lemma generalises the result above to processes with disasters, i.e.
$(1-p)\kappa\not\equiv0$. It delivers a dual process $\cX$ with respect to the pgf and
thus corresponds to Lemma \ref{lem:arb-off-dualilty} in the proof of Theorem
\ref{thm:bp-arb.off.dis}.

\begin{lemma}[A stronger duality]\label{lem:inhom-duality}
  Let $\log\frac10=-\infty$, $1/0=\infty$ and $1/\infty=0$. Then, for $x\in[0,1]$, $k\geq0$
  and $\cD_\infty:=\sigma(D_s;s\geq0)$, it holds that
  \begin{align*}
    \E_k[(1-x)^{Z_t}|\cD_\infty]
     &= (1-X_t)^k
  \end{align*}
  for a piecewise deterministic process $\cX=(X_t)_{t\geq0}$ given by
  \begin{align}
    X_t^{-1}
     &= \tfrac1xe^{-L_t}
       + \int_0^te^{-L_s}b_sds,
     \label{eq:inhom-X-explicit}
  \end{align}
  where
  $\displaystyle
    L_t = \int_0^t\big(b_s-d_s\big)ds - \int_0^t\log\Big(\frac1{p_s}\Big)dD_s
  $.
\end{lemma}
\begin{proof}
  Let $t$ be fixed, $G_t(x):=\E_k[(1-x)^{Z_t}|\cD_\infty]$,
  $\tau_0:=0$ and $\tau_1,\tau_2,\ldots$ be the jump times of $(D_t)$,
  i.e. the disaster times of $\cZ$. Note that the binomial disasters
  provide (with the left-side limit
  $Z_{\tau_n-} := \lim_{s\uparrow \tau_n} Z_s$), {\small
  \begin{align*}
    \E[(1-x)^{Z_{\tau_n}}|Z_{\tau_n-}=z] 
     &= \sum_{\ell=0}^z\genfrac(){0pt}{1}z\ell p_{\tau_n}^\ell(1-p_{\tau_n})^{z-\ell}(1-x)^\ell
      = \big(p_{\tau_n}(1-x) + (1-p_{\tau_n})\big)^z
      = (1-p_{\tau_n}x)^z.
  \end{align*}
  }
  \noindent
  Iterating this and Lemma \ref{lem:inhom-dual-no.dis.}, we obtain\\[1.5em]
  $\displaystyle
     G_t(x)
       = \E\Big[\E\big[(1-x)^{Z_t}\big|Z_{\tau_{D_t}},\cD_\infty\big]\Big|\cD_\infty\Big]
       = \E[(1-s_{D_t})^{Z_{\tau_{D_t}}}|\cD_\infty]
  $
  \begin{align*}
     &= G_{\tau_{D_t}}(s_{D_t})
       &\text{with }
         s_{D_t}^{-1}
          &= \frac{e^{v(\tau_{D_t})-v(t)}}x+e^{v(\tau_{D_t})}\int\limits_{\tau_{D_t}}^tb_se^{-v(s)}ds\\[1em]
     &= G_{\tau_{D_t}-}(p_{\tau_{D_t}}s_{D_t})=\ldots\\[1em]
     &= G_{\tau_{D_t-1}}(s_{D_t-1})
       &\text{with }
         s_{D_t-1}^{-1}
          &= \frac{e^{v(\tau_{D_t-1})-v(\tau_{D_t})}}{p_{\tau_{D_t}}s_{D_t}}
            + e^{v(\tau_{D_t-1})}\int\limits_{\tau_{D_t-1}}^{\tau_{D_t}}b_se^{-v(s)}ds\\[1em]
     &= G_{\tau_{D_t-2}}(s_{D_t-2})
       &\text{with }
         s_{D_t-2}^{-1}
          &= \frac{e^{v(\tau_{D_t-2})-v(\tau_{D_t-1})}}{p_{\tau_{D_t-1}}s_{D_t-1}}
            + e^{v(\tau_{D_t-2})}\int\limits_{\tau_{D_t-2}}^{\tau_{D_t-1}}b_se^{-v(s)}ds\\[1em]
     &=\ldots= G_{\tau_0}(s_0)
       &\text{with }
         s_0
          &= \frac{e^{v(\tau_0)-v(\tau_1)}}{p_{\tau_1}s_1}
            + e^{v(\tau_0)}\int\limits_{\tau_0}^{\tau_1}b_se^{-v(s)}ds\\[1em]
     &= (1-s_0)^k.
  \end{align*}
  Now, solving the recursion,
  \begin{align*}
    s_0^{-1}
     &= \Bigg(\ldots\Bigg(\frac1xe^{v(\tau_{D_t})-v(t)}
       + e^{v(\tau_{D_t})}\int_{\tau_{D_t}}^tb_ye^{-v(y)}dy\Bigg)\\[1em]
     &\qquad\qquad\cdot\frac1{p_{\tau_{D_t}}}e^{v(\tau_{D_t-1})-v(\tau_{D_t})}
       + e^{v(\tau_{D_t-1})}\int_{\tau_{D_t-1}}^{\tau_{D_t}}b_ye^{-v(y)}dy\Bigg)\cdots\Bigg)\\[1em]
     &\qquad\qquad\cdot\frac1{p_{\tau_1}}e^{v(\tau_0)-v(\tau_1)}
       + e^{v(\tau_0)}\int_{\tau_0}^{\tau_1}b_ye^{-v(y)}dy\Bigg)\\[1em]
     &= \frac1xe^{-v(t)}\prod_{k=1}^{D_t}p_{\tau_k}^{-1}
       + \sum_{k=0}^{D_t}\prod_{\ell=1}^{k}p_{\tau_\ell}^{-1} \cdot
           e^{v(\tau_0)}\int_{\tau_k}^{\tau_{k+1}\wedge t}b_ye^{-v(y)}dy,\\[1em]
\intertext{where the empty product equals 1. (Then, for $D_t=0$ and thus $t\leq\tau_1$, one
           obtains the deterministic dual from Lemma \ref{lem:inhom-dual-no.dis.}.) Letting
           $\beta(t):=\int_0^tb_se^{-v(s)}ds$ and considering that $v(\tau_0)=v(0)=0$,
           the sum equates to}
     &\sum_{k=0}^{D_t}\prod_{\ell=1}^kp_{\tau_\ell}^{-1}
        \big(\beta(\tau_{k+1}\wedge t)-\beta(\tau_{k})\big)
      = \int_0^t\prod_{\ell=1}^{D_s}p_{\tau_\ell}^{-1}\cdot b_se^{-v(s)}ds.
  \end{align*}
  With
  \begin{align*}
    \prod_{k=1}^{D_t}p_{\tau_k}^{-1}
     &= \exp\Big(\sum_{s\leq t}\log\tfrac1{p_s}\cdot(D_s-D_{s-})\Big)
      = \exp\Big(\int_0^t\log\tfrac1{p_s}dD_s\Big),
  \end{align*}
  it is simple to deduce that $s_0$ equals $X_t$ from $\eqref{eq:inhom-X-explicit}$
  and the proof is done.
\end{proof}

\begin{remark}\label{rem:lem-inhom-dual}\ 
  \begin{enumerate}
    \item
      This Lemma holds for arbitrary counting processes $(D_t)_{t\geq0}$.
      One might even consider a process with multiple jumps, e.g. $\Pw(\tau_k=0)>0$
      for some $k$.
    \item The process $\cX$ here is not of the form required for
      Corollary \ref{cor:p-jump-concave} or Theorem
      \ref{thm:convergence-$p$-jump-pros}, even if we choose constant
      $b,d,\kappa,p$ to obtain homogeneity: $X_t$ jumps from a state
      $1/(\frac ax + b)$ to $1/(\frac a{px}+b)$, which is not a
      $p$-jump. However, in the homogeneous case, letting
      $b\equiv\vartheta>0$, $d\in\R_+$, $\delta:=b-d>0$ and
      $\kappa\equiv1$, we can see that $X_t$ equates to
      $\overline W_t$, the time-reversal of $W_t$ in
      \eqref{eq:W-explicit} we used in the proof of Theorem
      \ref{thm:convergence-$p$-jump-pros}. Similarly one obtains that
      the (homogeneous) time-reversal $\overline X_t$ has the
      generator, for $f\in \mathcal C^1([0,1])$,
      \begin{align*}
        \cG_{\overline{\cX}}f(x)
         &= \kappa(f(px)-f(x)) + (bx(1-x)-dx)f'(x).
      \end{align*}
    \item
      The relationship between $\cX$ and $\cZ$ can be viewed as a \emph{stronger} duality,
      since the duality relation \eqref{eq:dual}, from the beginning of Section \ref{sec:bps},
      here does not only hold in expectation, but even in conditional expectation. (Taking
      expectation, \eqref{eq:dual} follows.)
  \end{enumerate}
\end{remark}

\noindent
Although we are not able to use Corollary \ref{cor:p-jump-concave} here, from
the previous Lemma we immediately obtain the following

\begin{proposition}\label{prop:inhom-limit}
  Let $L_t$ as in Lemma \ref{lem:inhom-duality} and $I_t:=\int_0^te^{-L_s}b_sds$.
  Supposing that $L_t\xrightarrow{t\to\infty}L\in[-\infty,\infty]$ almost surely
  and letting $I:=\lim_{t\to\infty}I_t$, there are 3 possible outcomes for the limit
  of $\E_k[(1-x)^{Z_t}|\cD_\infty]$:
  \begin{align*}
    \lim_{t\to\infty}\E_k[(1-x)^{Z_t}|\cD_\infty](\omega)
     &= \begin{cases}
          1                    &\text{if }\omega\in\{L=-\infty\}\cup\{I=\infty\}.\\[1em]
          (1-I(\omega)^{-1})^k &\text{if }\omega\in\{L=\infty\}\cap\{I<\infty\}.\\[1em]
          \big(1-\frac x{\exp(-L(\omega))+xI(\omega)}\big)^k
                               &\text{if }\omega\in\{L\in\R\}\cap\{I<\infty\}.
        \end{cases}
  \end{align*}
  The third case occurs if and only if
  \begin{align}
    \int_0^\infty(b_s+d_s)ds<\infty
    \quad\text{and}\quad
    \prod_{k\geq1}p(\tau_k(\omega))>0.
    \label{eq:inhom.condition-dep-on-x}
  \end{align}
\end{proposition}
\begin{proof}
  By construction and monotonicity of $(I_t)$, these three cases cover all possible outcomes.
  The results follow by insertion into Lemma \ref{lem:inhom-duality}.
  In case \emph{3.} there is $m(\omega)<\infty$ such that $m(\omega)\geq e^{L_t(\omega)}$
  for all $t$, since $(L_t(\omega))_t$ converges in $\R$. Thus, almost surely
  \begin{align*}
    \int_0^\infty b_sds
     &\leq\int_0^\infty me^{-L_s}b_sds
      = mI
      < \infty.
  \end{align*}
  Now, the convergence of $(L_t(\omega))$ and the non-negativity of $d$ and $\log\frac1p$
  give us that also $\int_0^\infty d_sds$ as well as
  \begin{align*}
    \int_0^\infty\log\tfrac1{p_s}dD_s(\omega)
     &= -\sum_{k\geq1}\log p_{\tau_k(\omega)}
  \end{align*}
  have to be finite, which shows that condition \eqref{eq:inhom.condition-dep-on-x} is
  necessary for the third case. To see the sufficiency, from \eqref{eq:inhom.condition-dep-on-x}
  the finiteness of $L(\omega)$ immediately follows analogously. Then, $(e^{-L_t(\omega)})_t$
  is bounded and thus by finiteness of $\int_0^\infty b_sds$, also $I(\omega)$ has to be finite.
\end{proof}

\begin{remark}\label{rem:prop:inhom-limit}\ 
  \begin{enumerate}
    \item
      The first part of condition \eqref{eq:inhom.condition-dep-on-x} implies that with probability 1
      there is only a finite number of birth and death events, while the second part offers either
      the possibility of $\lim_tD_t<\infty$ or $p$ converging to 1 on the support of $\kappa$,
      fast enough to compensate for $(D_t)$. 
    \item
      Only in the third case, the limiting probability generating function depends on $x$, which
      implies that, as soon as $b$ is bounded away from 0, $\cZ$ either goes extinct or explodes.
    \item
      Since this Proposition provides results depending directly on the paths of $(D_t)$, $b$, $d$
      and $p$, it can easily be applied to random environments in the sense of choosing $b$, $d$
      and/or $p$ to be stochastic processes.
    \item
      Another possible generalization could be to drop the assertion that the limit $L$ exists.
      Then, we see that the first case still only holds if $\limsup_tL_t=-\infty$ or $I=\infty$.
      Secondly, in the case of $I<\infty$ we still obtain a limit independent of $x$, only if
      $\liminf_tL_t=\infty$. Hence, only the third case changes, where we obtain bounds on the
      limit in terms of $\liminf_tL_t$ and $\limsup_tL_t$.
  \end{enumerate}
\end{remark}


\begin{proof}[Proof of Theorem \ref{thm:inhom-bdp}]\ \\
  First note that the assertions and Theorem \ref{thm:asymptotics-of-D-integral}
  imply that almost surely
  \begin{align*}
    L_t
      = \int_0^t(b_s-d_s)ds - \int_0^t\log\Big(\frac1{p_s}\Big)dD_s
     &\sim \iota h(t).
  \end{align*}
  (Since $h(t)=\Omega(t^\alpha)$ for some $\alpha>0$, in the case where
  $\beta\leq-1$, $\int_0^t\log(1/p_s)dD_s$ has either a finite limit or it
  lies in $O(t^{\alpha/2})\subset o(t^\alpha)$ such that in either case
  it does not contribute to the asymptotics of $L_t$. Otherwise, it is
  asymptotically equivalent to $\int_0^t\log(1/p_s)\kappa_sds$.)\\
  Now, we can apply Proposition \ref{prop:inhom-limit}:\\[0.5em]
  \emph{1.:} If $\iota=1$, $L_t\xrightarrow{t\to\infty}\infty$.
  Also, for almost every $\omega$ there is a $T(\omega)\in(0,\infty)$ such that
  $L_t(\omega)\geq(1-\varepsilon)h(t)$ for all $t\geq T(\omega)$. Thus,
  \begin{align*}
    I
     &\leq I_T + \int_T^\infty e^{-(1-\varepsilon)h(s)}b_sds
      < \infty.
  \end{align*}
  Hence, the second case of Proposition \ref{prop:inhom-limit} concludes that
  \begin{align*}
    \Pw(Z_t\xrightarrow{t\to\infty}0)
     &= \lim_{t\to\infty}\E_k[(1-1)^{Z_t}]
      = \E[(1-I^{-1})^k]
      < 1.
  \end{align*}
  \emph{2.:} If $\iota=-1$, it is clear, that $L_t\xrightarrow{t\to\infty}-\infty$
  and independently of the integral condition of \emph2. the first part of
  Proposition \ref{prop:inhom-limit} concludes
  \begin{align*}
    \Pw(Z_t\xrightarrow{t\to\infty}0)
     &= \E_k[\lim_{t\to\infty}(1-1)^{Z_t}]
      = 1.
  \end{align*}
  Otherwise, i.e. if $\iota=1$ but the integral condition holds,
  $L_t\xrightarrow{t\to\infty}\infty$ and analogously to \emph{1.} we
  find a finite random variable $T'$ such that $L_t\leq(1+\varepsilon)h(t)$
  for all $t\geq T'$ almost surely. Thus,
  \begin{align*}
    I
     &= \int_0^\infty e^{-L_s}b_sds
      \geq I_{T'} + \int_{T'}^\infty e^{-(1+\varepsilon)h(s)}b_sds
      = \infty.
  \end{align*}
  Then again, the first part of Proposition \ref{prop:inhom-limit} concludes the proof.
\end{proof}

\begin{remark}[Normalization function, rates of convergence]\label{rem:inhom-conv-rates}
  \begin{enumerate}
  \item There are two major cases in which a normalisation function
    $h$ as in Theorem \ref{thm:inhom-bdp} does not exist:
    \begin{enumerate}
    \item
      The integral $\ell(t):=\int_0^t (b_s - d_s - \kappa_s\log(\frac1{p_s}))ds$
      converges to a constant. Then, $\cZ$ will exhibit only a finite number
      of birth events almost surely and converge to a random variable, where
      the third part of Proposition \ref{prop:inhom-limit} provides a way to
      compute the limiting distribution.
    \item
      The integral $\ell$ oscillates too strongly -- e.g. $\ell(t) = t(1+\sin(t))$.
      This might happen in periodic models, which were briefly discussed in
      \cite{Kendall1948}. In this case, Lemma \ref{lem:inhom-duality} still
      holds, while Proposition \ref{prop:inhom-limit} as well as Theorem
      \ref{thm:asymptotics-of-D-integral} do not apply.
    \end{enumerate}
  \item In case \emph{2.} of Theorem \ref{thm:inhom-bdp}, for the
    convergence rates of the survival probability conditioned on
    $\cD_\infty$, the $\sigma$-algebra of the disaster times, we can
    estimate for arbitrary $k\geq1$, using the processes $(X_t)$ and
    $(L_t)$ from Lemma \ref{lem:inhom-duality} with $X_0=1$ and
    Bernoulli's inequality
    \begin{align*}
      -\frac1{h(t)}\log\Pw_k(Z_t>0|\cD_\infty)
      &= -\frac1{h(t)}\log\big(1-(1-X_t)^k\big)\\[1em]
      &\sim -\frac1{h(t)}\log(X_t)
        = \frac1{h(t)}\log\Big(e^{-L_t} + \int_0^te^{-L_s}b_sds\Big)\\[1em]
      &\sim \frac1{h(t)}\max\Bigg\{-L_t,\ \log\Big(\int_0^te^{-L_s}b_sds\Big)\Bigg\},
    \end{align*}
    where $\sim$ denotes asymptotic equivalence, i.e. $f\sim g\Leftrightarrow f(t)/g(t)\to1$
    as $t\to\infty$. Then, by Theorem \ref{thm:asymptotics-of-D-integral}, $\frac{L_t}{h(t)}\to\iota$,
    while for $0<\delta<\varepsilon$ and $t\geq t_0$ large enough it holds
    \begin{align*}
      \int_{t_0}^te^{-(1+\delta)h(s)}b_sds
      &\leq \int_{t_0}^te^{-L_s}b_sds
        \leq \int_{t_0}^te^{-(1-\delta)h(s)}b_sds.
    \end{align*}
    With more knowledge on $h$ and $b$, this approach can be used to compute bounds
    on the convergence rates.
  \end{enumerate}
\end{remark}

\appendix

\section{Large Deviations}\label{sec:ldp}

In Theorem \ref{thm:convergence-$p$-jump-pros}, we make use of large deviations for
Poisson processes. These can be e.g.\ read from \cite{DemboZeitouni1998}, Exercise
5.2.12.

\begin{lemma}[Large deviations for a Poisson process]\label{lem:ld-PP}
  Let $P =(P_t)_{t\geq 0}$ be a unit rate Poisson process.
  \begin{align}\label{lld:ld1}
    \lim_{t\to\infty} \tfrac 1t \log \mathbb P(P_t \geq xt) &= -(1 - x + x\log x) \text{ for }x>1,\\
    \label{lld:ld2}
    \lim_{t\to\infty} \tfrac 1t \log \mathbb P(P_t \leq xt) &= -(1 - x + x\log x)  \text{ for }x\in(0,1).
  \end{align}
  Moreover, for any $x\in(0,1)$
  \begin{align}\label{lld:ld3}
    & \lim_{t\to\infty} \tfrac 1t \log \mathbb P( P_s\leq xs \text{ for all }s\leq t) 
      = -(1 - x + x\log x).
  \end{align}
\end{lemma}

\begin{proof}
  The first two assertions, \eqref{lld:ld1} and \eqref{lld:ld2} are a consequence of Cr\'amer's
  Theorem \cite[Theorem 2.2.3, p.27]{DemboZeitouni1998}. Moreover, the large deviation result
  \eqref{lld:ld3} is an application of \cite{DemboZeitouni1998}, Exercise 5.2.12, by rescaling
  and choosing $\phi(t)=xt$.
\end{proof}


\begin{thebibliography}{}

\bibitem[Athreya and Kaplan, 1976]{AthreyaKaplan1976}
Athreya, K.~B. and Kaplan, N. (1976).
\newblock Limit theorems for a branching process with disasters.
\newblock {\em Journal of Applied Probability}, 13(3):466--475.

\bibitem[Athreya and Ney, 1972]{AthreyaNey1972}
Athreya, K.~B. and Ney, P.~E. (1972).
\newblock {\em Branching Processes}.
\newblock Springer.

\bibitem[Bansaye et~al., 2013]{Bansaye2013}
Bansaye, V., Carlos, J., Millan, P., and Smadi, C. (2013).
\newblock On the extinction of continuous state branching processes with
  catastrophes.
\newblock {\em Elec. J. Probab.}, 18(106):1--31.

\bibitem[Bartoszynski et~al., 1989]{Bartoszynski1989}
Bartoszynski, R., Biihler, W.~J., Chan, W., and Pearl, D.~K. (1989).
\newblock Population processes under the influence of disasters occurring
  independently of population size.
\newblock {\em J. Math. Biol.}, 27:167--178.

\bibitem[Bladt and Nielsen, 2017]{BladtNielsen2017}
Bladt, M. and Nielsen, B.~F. (2017).
\newblock {\em Regeneration and Harris Chains}, pages 387--435.
\newblock Springer US, Boston, MA.

\bibitem[Brockwell, 1985]{Brockwell1985}
Brockwell, P.~J. (1985).
\newblock The extinction time of a birth, death and catastrophe process and of
  a related diffusion model.
\newblock {\em Advances in Applied Probability}, 17(1):42--52.

\bibitem[Brockwell et~al., 1982]{Brockwell1982}
Brockwell, P.~J., Gani, J., and Resnick, S.~I. (1982).
\newblock Birth, immigration and catastrophe processes.
\newblock {\em Advances in Applied Probability}, 14(4):709--731.

\bibitem[B{\"u}hler and Puri, 1989]{BühlerPuri1989}
B{\"u}hler, W.~J. and Puri, P.~S. (1989).
\newblock The linear birth and death process under the influence of
  independently occurring disasters.
\newblock {\em Probability Theory and Related Fields}, 83(1):59--66.

\bibitem[Casanova et~al., 2016]{CasanovaEtAl2016}
Casanova, A.~G., Kurt, N., Wakolbinger, A., and Yuan, L. (2016).
\newblock An individual-based model for the lenski experiment, and the
  deceleration of the relative fitness.
\newblock {\em Stochastic Processes and their Applications}, 126(8):2211 --
  2252.

\bibitem[Chaganty and Sethuraman, 1993]{ChagantySethuraman1993}
Chaganty, N.~R. and Sethuraman, J. (1993).
\newblock Strong large deviation and local limit theorems.
\newblock {\em The Annals of Probability}, 21(3):1671--1690.

\bibitem[Dembo and Zeitouni, 1998]{DemboZeitouni1998}
Dembo, A. and Zeitouni, O. (1998).
\newblock {\em Large Deviations Techniques and Applications}.
\newblock Applications of mathematics. Springer.

\bibitem[Etheridge, 2001]{Etheridge2001}
Etheridge, A. (2001).
\newblock {\em An introduction to superprocesses}.
\newblock American Mathematical Society.

\bibitem[Ethier and Kurtz, 1986]{EK86}
Ethier, S.~N. and Kurtz, T.~G. (1986).
\newblock {\em Markov processes. Characterization and convergence}.
\newblock Wiley Series in Probability and Mathematical Statistics: Probability
  and Mathematical Statistics. John Wiley \& Sons Inc., New York.

\bibitem[Feller, 1971]{Feller1971}
Feller, W. (1971).
\newblock {\em An introduction to probability theory and its applications.
  {V}ol. {II}.}
\newblock Second edition. John Wiley \& Sons Inc., New York.

\bibitem[Harris, 1963]{Harris1963}
Harris, T.~E. (1963).
\newblock {\em The theory of branching processes}.
\newblock Die Grundlehren der Mathematischen Wissenschaften, Bd. 119.
  Springer-Verlag, Berlin.

\bibitem[Jansen and Kurt, 2014]{JansenKurt2014}
Jansen, S. and Kurt, N. (2014).
\newblock On the notion(s) of duality for {M}arkov processes.
\newblock {\em Probability Surveys}, 11:59--120.

\bibitem[Kaplan et~al., 1975]{KaplanEtAl1975}
Kaplan, N., Sudbury, A., and Nilsen, T.~S. (1975).
\newblock A branching process with disasters.
\newblock {\em Journal of Applied Probability}, 12(1):47--59.

\bibitem[Kendall, 1948]{Kendall1948}
Kendall, D.~G. (1948).
\newblock On the generalized 'birth-and-death' process.
\newblock {\em Ann. Math. Statist.}, 19(1):1--15.

\bibitem[Kumar et~al., 1998]{KumarII}
Kumar, B.~K., Vijayakumar, A., and Thilaka, B. (1998).
\newblock Multitype branching processes with disasters ii: Total sojourn time
  and number of deaths.
\newblock {\em Mathematical and Computer Modelling}, 28(11):103--114.

\bibitem[Lambert, 2008]{Lambert2008}
Lambert, A. (2008).
\newblock Population dynamics and random genealogies.
\newblock {\em Stochastic Models}, 24:45--163.

\bibitem[Pakes, 1986]{Pakes1986}
Pakes, A.~G. (1986).
\newblock The markov branching-castastrophe process.
\newblock {\em Stochastic Processes and their Applications}, 23(1):1 -- 33.

\bibitem[Pakes and Pollett, 1989]{Pakes1989}
Pakes, A.~G. and Pollett, P. (1989).
\newblock The supercritical birth, death and catastrophe process: limit
  theorems on the set of extinction.
\newblock {\em Stochastic Processes and their Applications}, 32(1):161 -- 170.

\bibitem[Peng et~al., 1993]{Peng1993}
Peng, N., Pearl, D.~K., Chan, W., and Bartoszyński, R. (1993).
\newblock Linear birth and death processes under the influence of disasters
  with time-dependent killing probabilities.
\newblock {\em Stochastic Processes and their Applications}, 45(2):243 -- 258.

\bibitem[Polfeldt, 1969]{Polfeldt1969}
Polfeldt, T. (1969).
\newblock Integrating regularly varying functions with exponent -1.
\newblock {\em SIAM Journal on Applied Mathematics}, 17(5):904--908.

\bibitem[Seneta, 1976]{Seneta1976}
Seneta, E. (1976).
\newblock {\em Regularly Varying Functions}.
\newblock Number Nr. 508 in Lecture Notes in Mathematics. Springer-Verlag.

\bibitem[Thilaka et~al., 1998]{Thilaka1998}
Thilaka, B., Kumar, B., and Vijayakumar, A. (1998).
\newblock Multitype branching processes with disasters {I}: The number of
  particles in the system.
\newblock {\em Mathematical and Computer Modelling}, 28(11):87 -- 102.

\end{thebibliography}

\end{document}